\documentclass[leqno]{amsart}
\input{cyracc.def}
\font\cyr=wncyi8
\usepackage{hyperref} 

\catcode`\@=11
\def\@evenfoot{\rule{0pt}{20pt}[\today] \hfill [{\tt \jobname.tex}]}
\def\@oddfoot{\rule{0pt}{20pt}{[\tt \jobname.tex}]\hfill [\today]}
\catcode`\@=13

\usepackage{ adjustbox }
\usepackage{ amsmath }
\usepackage{ amssymb }
\usepackage{ amsthm }
\usepackage{ array }
\usepackage{ bbold }
\usepackage{ booktabs }
\usepackage{ cite }
\usepackage{ enumitem }
\usepackage{ euscript }
\usepackage{ latexsym }
\usepackage[pagewise]{ lineno }
\usepackage{ mathrsfs }
\usepackage{ url }
\usepackage{ xcolor }
\usepackage[all]{ xy }

\newtheorem{theorem}{Theorem}[section]
\newtheorem{corollary}[theorem]{Corollary}
\newtheorem{lemma}[theorem]{Lemma}

\newtheorem*{problem}{Problem}

\theoremstyle{definition}
\newtheorem{notation}[theorem]{Notation}
\newtheorem{definition}[theorem]{Definition}
\newtheorem{example}[theorem]{Example}
\newtheorem{nonexample}[theorem]{Non-example}
\newtheorem{remark}[theorem]{Remark}
\newtheorem*{acknowledgment}{Acknowledgment}

\newcommand\hksqrt[1]{
\setbox0=\hbox{$\sqrt{#1\,}$}
\dimen0=\ht0
\advance\dimen0-0.4\ht0
\setbox2=\hbox{\vrule height\ht0 depth-\dimen0}
{\box0\lower0.4pt\box2}
}

\def\bfk{{\mathbb k}}

\def\op{{\rm op}}

\def\bbk{{\mathbb k}}
\def\oP{{\EuScript{P}}}

\def\oA{{\EuScript{A}}}
\def\cald{{\mathcal D}}
\def\calf{{\mathcal F}}
\def\calP{\oP}
\def\ttC{{\tt C}}
\def\prez#1#2{{\langle\, #1;#2 \,\rangle}}
\def\dl#1#2{#1(#2)\hbox{\;$\rightsquigarrow$\;}#2(#1)}

\def\Com{\mbox{{$\mathcal C$}\hskip -.3mm {\it om}}}
\def\Lie{\hbox{{$\mathcal L$}{\it ie\/}}}
\def\Mag{\hbox{{$\EuScript M$}{\it ag\/}}}
\def\Lei{\hbox{{$\mathcal L$}{\it ei\/}}}
\def\Ass{\mathcal A{\it ss}}
\def\uAss{\underline{\hbox{$\mathcal A \hskip -.0em ss$}}}
\def\BB{\mathcal{BB}}
\def\uBB{\underline{\hbox{$\mathcal{BB} \hskip -.0em$}}}

\def\b{{\bullet}}
\def\c{{\circ}}

\colorlet{RED}{red}

\makeatletter
\def\l@section{\@tocline{1}{0pt}{0pc}{}{}}
\def\l@subsection{\@tocline{2}{0pt}{1pc}{4.6em}{}}
\def\l@subsubsection{\@tocline{3}{0pt}{1pc}{7.6em}{}}
\renewcommand{\tocsection}[3]
{\indentlabel{\@ifnotempty{#2}{\makebox[1.25em][l]{\hfil\ignorespaces#1#2.}}}#3}
\renewcommand{\tocsubsection}[3]
{\indentlabel{\@ifnotempty{#2}{\hspace*{1.25em}\makebox[2.5em][l]{\hfil\ignorespaces#1#2.}}}#3}
\renewcommand{\tocsubsubsection}[3]
{\indentlabel{\@ifnotempty{#2}{\hspace*{3.25em}\makebox[2.75em][l]{\ignorespaces#1#2.}}}#3}
\makeatother

\allowdisplaybreaks

\title[Distributive laws between the Three Graces]
{%
Distributive laws between the Three Graces
}

\author{Murray Bremner}

\address{Department of Mathematics and Statistics, University of Saskatchewan, Canada}

\email{bremner@math.usask.ca}

\author{Martin Markl}

\address{Institute of Mathematics of the Czech Academy of Sciences, Prague, Czech Republic}

\email{markl@math.cas.cz}

\dedicatory{All algebras are equal, but some algebras are more equal than others.}

\thanks{Murray Bremner was supported by the Discovery Grant \emph{Algebraic Operads} from NSERC,
the Natural Sciences and Engineering Research Council of Canada.
Martin Markl was supported by the Eduard \v Cech Institute
P201/12/G028, grant  GA \v CR 18-07776S and RVO: 67985840.}

\begin{document}
\bibliographystyle{plain}

\begin{abstract}
By the Three Graces we refer, following J.-L. Loday, to the algebraic operads
$\Ass$, $\Com$, and $\Lie$, each generated by a single binary operation; algebras
over these operads are respectively associative, commutative
associative, and Lie.
We classify all distributive laws (in the categorical sense 
of Beck) between these three operads. 
Some of our results depend on the computer algebra system Maple, 
especially its packages \texttt{LinearAlgebra} and \texttt{Groebner}.
\end{abstract}
 


\maketitle

\setcounter{tocdepth}{1}
\small \tableofcontents


\section{Introduction}

As the epigraph indicates%
\footnote{\,The allusion is to a famous quotation from George Orwell's satire 
\emph{Animal Farm}.}, some algebras are more important than others. 
Experience teaches us that the most common classes of algebras are the Three Graces%
\footnote{\,This terminology originated with J.-L. Loday, referring in particular to 
the famous painting \emph{Les Trois Gr\^aces}, a Renaissance masterpiece 
by Lucas Cranach the Elder.
Since 2011 it has been in the collection of the Mus\'ee du Louvre in Paris.
It depicts the \emph{charites} or daughters of Zeus from classical Greek mythology:
Aglaea (meaning elegance or splendor), Euphrosyne (mirth or happiness), and Thalia 
(youth or beauty).} 
--- associative, commutative associative, and Lie ---
together with other classes of algebras that combine these in a specific way.
The algebras in these three classes are representations of 
the quadratic Koszul operads denoted
$\Ass$, $\Com$, and $\Lie$, or created from these operads using 
quadratic homogeneous distributive laws 
(the precise meaning of this phrase will be explained in \S\ref{sectiondistributivelaws}).   
Examples of structures combining two of these operads are the following:
\begin{itemize}[leftmargin=*]
\item 
Poisson algebras, omnipresent in classical mechanics \cite{BFFLS1,BFFLS2,K2,K3,KM,LL,MR}
\item 
Gerstenhaber algebras \cite{A2,DTT,G,K1,LZ}
\item
Batalin-Vilkovisky algebras \cite{A1,GTV,Getzler,H,R,X}
\item 
$e_n$-algebras and the little cubes operad from homotopy theory \cite{BV,CG,F,May,S1,S2}
\end{itemize}
The motivation for the present article is to investigate whether there are other combinations
of the Three Graces via such a distributive law, beyond the well-known examples.
It turned out that there are, up to isomorphism, only the classical, well-known
distributive laws, plus the trivial and truncated ones.  
Since classifying distributive laws amounts to solving hundreds of
quadratic equations, we found it fascinating that for the Three Graces this huge
system has only a small finite number of solutions. 
This kind of rigidity which the Three Graces possess might be another reason why 
they are more equal than others.
Although the results of this article might not surprise everyone,
we thought that at some point of the history of mankind this analysis
had to be made%
\footnote{
In the context of the present paper we found it interesting that, according to \cite{GV}, 
one of the Three Graces --- the operad $\Lie$ --- has the property that the variety
of its algebras is the only variety of non-associative algebras which is locally algebraically 
cartesian closed.
}.

The existence of this paper was greatly facilitated by advances in computer-assisted mathematics, 
and in particular the computer algebra system Maple; worksheets written by the first author 
expressly for this project were used to extend hand calculations of the second author dating from 
some 20 years ago. 

In Section~\ref{sectiondistributivelaws} we recall Jon Beck's definition of 
distributive laws~\cite{B} along with its operadic translation \cite{FM,M}. 
In the subsequent sections we classify all homogeneous operadic distributive laws 
between the Three Graces.
The last section classifies distributive laws
between associative and magmatic multiplications. 
It points to the fact that, while outside the realm of the Three Graces various
bizarre-looking distributive laws exist, they may turn out to be
isomorphic to the expected ones.
Classifying all possible distributive laws is difficult, but to verify whether 
a given formula induces a distributive law is relatively simple. 
We did so by hand in Sections~\ref{Dnes_jsem_letal_na_205ce.} and~\ref{Pot_ze_mne_leje.}, 
believing it might elucidate the meaning of coherence of distributive
laws. 

Let us close this introduction by formulating

\begin{problem}
Characterize pairs of operads for which there exists only a finite number of
non-isomorphic distributive laws between them.
\end{problem}

Any two of the Three Graces form such a
pair as does, according to Section~\ref{Pot_ze_mne_leje.}, also the
pair of operads for associative and magmatic multiplications.
In a sequel to this paper we intend to perform a similar analysis for bialgebras. 

\begin{acknowledgment} 
We are indebted to Vladimir Dotsenko
for explaining to us that the Eulerian
substitution~(\ref{Po_navratu_z_Moskvy.}) brings one of our bizarre
distributive laws to the standard truncated one.
\end{acknowledgment} 


\section{Distributive laws}
\label{sectiondistributivelaws}


\subsection{Background}

In this section we recall basic facts about distributive laws, closely following the work of
Fox and the second author \cite{FM}; see also the original paper by Beck \cite{B} and the works 
of Street \cite{Street} and Lack \cite{Lack}.
We will assume working knowledge of operads and their various versions.
Suitable references are the
monographs~\cite{BD-book,markl-shnider-stasheff:book,loday-vallette} 
complemented
with~\cite{markl:handbook} and 
the original source~\cite{ginzburg-kapranov:DMJ94}. 
All algebraic objects will be defined over a ground field ${\bbk}$ of characteristic $0$, 
and the basic category will be the monoidal category of $\mathbb{Z}$-graded vector spaces
with the Koszul sign rule.
Loosely speaking, a distributive law relates operations of two types, in the sense that 
it rearranges multiple applications of these operations in such a way that operations of 
the first type are applied first, followed by those of the second type. 
Moreover, this rearrangement must be done in a way that is coherent in the categorical sense. 

\begin{example}
Poisson algebras have two operations: 
the Lie bracket $[a,b]$ and the commutative associative multiplication $a \cdot b$.
These operations are related by the derivation~law:
\begin{equation}
\label{derivationlaw}
[ a \cdot b, c ] = a \cdot [b,c] + [a,c] \cdot b.
\end{equation}
On the left side we see the operation of the second type, namely $a \cdot b$, multiplied by $c$ 
using the operation of the first type, while in each term on the right side we 
first apply the Lie bracket and then the operation of the second type. 
By repeated application of equation \eqref{derivationlaw} regarded as a directed (left to right) rewrite rule,
we may convert any monomial, involving some number of occurrences of the first and 
second operations, into a sum of terms where all of the Lie brackets have been applied first. 
Coherence means that equation \eqref{derivationlaw} does not introduce any `unexpected relations';
to be precise, this means that the free Poisson algebra generated by a vector space $X$ 
is naturally isomorphic \cite[Lemma 1]{Shestakov}
to the free commutative associative algebra on the free Lie algebra generated by $X$; 
symbolically,
\[
\mathbf{Pois}(X) \cong \mathbf{Com}( \mathbf{Lie}(X) ).
\]
Distributive laws are ordered: equation \eqref{derivationlaw} is a distributive law of a Lie multiplication 
over a commutative associative multiplication; we denote this by
\[
\cald : \dl\Lie\Com.
\]
\end{example}

\begin{definition}
\label{defdl}
Let us recall the precise definition introduced by Beck \cite{B}.
Assume that $T_1 = (T_1,\mu_1,\eta_1)$ and  $T_2 = (T_2,\mu_2,\eta_2)$ are monads (formerly called 
triples) on a category~$\ttC$.
A distributive law guarantees that for every $T_2$-algebra $A$ in $\ttC$, the object 
$T_2 (A) \in \ttC$ has the structure of a $T_1$-algebra in a very explicit way. 
More precisely, a \textit{distributive law} 
is a~natural transformation 
\begin{equation}
\label{distributivelaw}
\lambda: T_1T_2 \to T_2 T_1,
\end{equation}
such that, for every $T_2$-algebra $A = (A,\alpha\colon T_2( A) \to A)$, the object $T_2(A)\in \ttC$ 
is a $T_1$-algebra with structure morphism
\[
T_1 T_2 A \xrightarrow{\;\lambda\;} T_2 T_1 A  \xrightarrow{\;T_2\alpha\;} T_2 A.
\] 
This imposes certain conditions on $\lambda$ whose explicit form can be found in \cite{B}; 
see also \cite[\S3]{FM}.
In this situation, the endofunctor $T = T_2 T_1$ is again a monad, with structure 
transformations
\[
\mu = T_2\mu_1 \circ \mu_2 T_1^2 \circ T_2 \lambda T_1,
\qquad
\eta = \eta_1 \circ \eta_2 \circ T_1.
\]
The equality
\[
T(X) = T_2\big(T_1 (X)\big), \qquad X \in \ttC,
\]
may be interpreted as saying that the free $T$-algebra on $X$ is (as an object of $\ttC$) naturally 
isomorphic to the free $T_2$-algebra generated by the free $T_1$-algebra on $X$.
\end{definition}

\begin{example}
\label{exampledl}
We know one example of a distributive law from elementary school. 
If $\ttC$ is the category of sets, $T_1$ the commutative monoid monad, and $T_2$ the abelian group 
monad, then the equation $x(a+b) = xa + xb$ generates a natural transformation $T_1 T_2 \to T_2 T_1$ 
taking a product of sums to a sum of products. 
The algebras for the combined monad $T = T_2 T_1$ are commutative rings.
\end{example}


\subsection{Setting of this article}

We restrict ourselves, for reasons explained below, to monads given by the free
$\oP$-algebra functor for a quadratic finitely generated operad $\oP$. 
Moreover, the distributive laws we consider will be given by very specific data. 
Before we give a precise definition, we need to establish some notational conventions;
we write $\Sigma_n$ for the symmetric group on $n$ letters.

\begin{notation}
If $E$ is a vector space which is also a $\Sigma_2$-module, then $\calf(E)$ denotes the free operad 
generated by $E$ placed in arity $2$. 
For a subspace $R \subseteq \calf(E)(3)$, we write $\prez ER$ for the quotient $\calf(E)/(R)$ of
the free operad $\calf(E)$ modulo the operad ideal $(R)$ generated by $R$.  
\end{notation}

Suppose that the $\Sigma_2$-module $E$ has an invariant decomposition $E = E_1 \oplus E_2$. 
This induces the decomposition
\[
\calf(E)(3) = \calf(E)(3)_{11}\oplus \calf(E)(3)_{12} \oplus \calf(E)(3)_{21} \oplus \calf(E)(3)_{22},
\]
where $\calf(E)(3)_{ij}$ is the $\Sigma_3$-invariant subspace of $\calf(E)(3)$ generated by the 
compositions of the form $\mu(1,\nu)$ and $\mu(\nu,1)$ with $\mu \in E_i$ and $\nu \in E_j$ for 
$i,j = 1,2$. 
Notice that $\calf(E)(3)_{ii}$ can be identified with the image of the map $F(E_i)(3)\to \calf(E)(3)$ 
induced by the inclusion $E_i \subseteq E$.
Let us consider a
$\Sigma_3$-invariant map 
\begin{equation}
\label{sigma3map}
\cald\colon
\calf(E)(3)_{12} \longrightarrow \calf(E)(3)_{21}.
\end{equation}
Every such map defines a $\Sigma_3$-submodule $R_{\cald} \subseteq \calf(E)(3)$ generated by elements 
of the form $x - \cald(x)$ for $x \in \calf(E)(3)_{12}$.

Let $\oP=\prez ER$ be a quadratic operad for which there exists a $\Sigma_2$-module decomposition 
$E = E_1 \oplus E_2$, a $\Sigma_3$-equivariant linear map 
$\cald\colon \calf(E)(3)_{12} \to \calf(E)(3)_{21}$, 
and $\Sigma_3$-invariant subsets $R_i \subseteq \calf(E)(3)_{ii}$, $i=1,2$, 
such that $R = R_1 \oplus R_{\cald} \oplus R_2$. 
In other words, the operad $\oP$ has the presentation
\begin{equation}
\label{presentation}
\oP = \prez{E_1 \oplus E_2}{R_1 \oplus R_{\cald} \oplus R_2}.
\end{equation}
We consider the suboperads $\oP_i = \prez{E_i}{R_i} \subseteq \oP$
for $i = 1, 2$.
For $1 \le s \le l \le n$, and a~sequence $m_1, \dots, m_l \ge 1$ with $m_1 + \cdots + m_l = n$, 
we write $\oP(n)_l$ for the $\Sigma_n$-submodule of $\oP(n)$ generated by the elements of the form 
$\mu( \nu_1, \dots, \nu_l )$ for $\mu\in \oP_2(l)$ and $\nu_s \in \oP_1(m_s)$.
The inclusions $\oP_i \subseteq \oP$ ($i = 1, 2$) induce, for any $n
\ge 2$, an equivariant linear map
\[
\xi(n)\colon \bigoplus_{1 \le l \le n} \oP(n)_l \,\longrightarrow\, \oP(n).
\]

\begin{definition}
\label{defdl2}
We say that the map $\cald$ of equation \eqref{sigma3map} is an (operadic homogeneous quadratic) 
\textit{distributive law} of $\oP_1$ over $\oP_2$ if the map $\xi(n)$ is an isomorphism 
for every $n \ge 2$. 
We express this fact by writing $\cald\colon \dl{\oP_1}{\oP_2}$.
\end{definition}

We denote by $T_i$ ($i = 1, 2$) the free $\oP_i$-operad monad acting on the category of $\Sigma$-modules. 
From \cite[Proposition 2.6]{M} we know that a distributive law in the sense of Definition 
\ref{defdl2} determines, in a very explicit way, a distributive law \eqref{distributivelaw} in the sense of Beck,
namely $\lambda\colon T_1 T_2 \to T_2 T_1$, for which the combined monad $T = T_2T_1$ is the monad
for $\oP$-algebras. 
Of course, not all distributive laws in the sense of Beck are distributive laws in the sense of
Definition \ref{defdl2}: see Example \ref{exampledl}, which is not even `operadic'  since $x$ appears 
twice in the right hand side.   

\begin{remark}
One sometimes says more precisely that the map in~(\ref{sigma3map})
satisfying the condition of Definition~\ref{defdl2} is a {\em
  rewrite rule\/} defining a distributive law between the associated
monads. Rewrite rules are often conveniently expressed in the form of
an equation such as \eqref{derivationlaw} whose left hand side belongs to
$\calf(E)(3)_{12}$ and right hand side to $\calf(E)(3)_{21}$.
\end{remark}

The adjective \emph{quadratic} in Definition \ref{defdl2} means that the distributive law involves
quadratic operads and is therefore determined by its behavior inside $\calf(E)(3)$; from this it follows
that the resulting operad \eqref{presentation} is again quadratic. 
Quadratic operads have their Koszul duals, and therefore we have the following result.

\begin{lemma}
\label{duallemma}
\emph{\cite[Lemma 9.3]{FM}}
In the situation of Definition \ref{defdl2} one has the following canonical dual quadratic homogeneous 
distributive law of $\calP_2^!$ over $\calP_1^!$,
\[
\cald^! \colon \dl{\oP_2^!}{\oP_1^!},
\]
such that the resulting combined operad is the Koszul dual of the operad \eqref{presentation}.
\end{lemma}

The adjective \emph{homogeneous} in Definition \ref{defdl2} means that the distributive law preserves 
the bigrading of the free operad $\calf(E_1\oplus E_2)$ given by the number of operations first from 
$E_1$ and then from $E_2$. 
Therefore the resulting combined quadratic operad \eqref{presentation} is also bigraded, and hence 
free $\oP$-algebras are also bigraded. 
As a consequence, the operadic cohomology of $\oP$-algebras can be calculated as the cohomology of 
a bicomplex combining $\oP_1$- and $\oP_2$-cochains; see \cite[Theorem 10.2]{FM}.  

\begin{example}
An `archetypal' distributive law in the sense of Definition \ref{defdl2} is equation~\eqref{derivationlaw} 
which combines Lie and commutative associative algebras into Poisson algebras.
A particular \emph{inhomogeneous} quadratic operadic distributive law 
is that which describes associative algebras as algebras with two 
operations, a commutative nonassociative multiplication $-\cdot-$ and a Lie bracket $[-,-]$, with 
the relations
\[
[x,y\cdot z] = [x,y]\cdot z +y\cdot[x,z], 
\qquad\qquad
[y,[x,z]] = (x\cdot y)\cdot z-x\cdot (y\cdot  z).
\]
This law is indeed not homogeneous, 
since on the left side of the second equation we see a~term of bidegree $(0,2)$, 
i.e., with no instance of the multiplication $-\cdot-$ but two instances of $[-,-]$, 
while the terms on the right hand side are of bidegree $(2,0)$.
\end{example}

In general, defining a transformation $\lambda$ as in equation
\eqref{distributivelaw}, 
and verifying that it is indeed a~distributive law, is a difficult problem, 
\emph{operadic} distributive laws are determined by a very small set of data of 
essentially finitary nature.
Moreover, verifying the required property (Definition \ref{defdl2}) 
boils down to a finite calculation.

\begin{theorem}
\emph{\cite[Theorem 2.3]{M}}
\label{turmo}
The map $\xi(n)$ is an isomorphism for all $n \ge 2$ if and only if it is an isomorphism for $n = 4$.
\end{theorem}

It can also be shown that the maps $\xi(n)$ are epimorphisms for an arbitrary $\cald$ as in 
equation \eqref{sigma3map}. 
Since both the domain and codomain of $\xi(n)$ are finite dimensional, it is enough to verify that
\[
\dim \bigoplus_{1\leq l\leq 4} \! \oP(4)_l = \dim \oP(4).
\]
It is clear that this equation, when expressed in terms of structure constants, leads to a system 
of quadratic equations without constant terms.
In particular, taking $\cald$ to be identically zero always gives a distributive law, 
the \emph{trivial} one.

The discussion in this section makes clear the prominent r\^ole played by 
operadic homogeneous quadratic distributive laws. 
In the rest of this article we will deal exclusively with such distributive laws, 
and will therefore omit the adjectives \emph{operadic homogeneous quadratic}
and speak simply about \emph{distributive laws}.


\subsection{Case studies}

In the following sections we describe all distributive laws between the Three Graces. 
It suffices to consider the seven cases 
in the first column of the following table,
since the dual cases in the second column follow by Lemma~\ref{duallemma}: 

\begin{center}
\begin{tabular}{l@{\qquad\qquad}l}
distributive law & Koszul dual
\\ \midrule
$\dl\uAss\uAss$ & self-dual
\\
$\dl{\Ass}{\Ass}$ & self-dual
\\
 $\dl\Lie{\Ass}$ &$\dl{\Ass}\Com$ 
\\
 $\dl\Com{\Ass}$ & $\dl{\Ass}\Lie$ 
\\
$\dl\Com\Com$ & $\dl\Lie\Lie$
\\
$\dl\Com\Lie$ & self-dual
\\
$\dl\Lie\Com$ & self-dual
\end{tabular}
\end{center}


\section{Distributive laws $\dl{\Ass}{\Ass}$}
\label{Dnes_jsem_letal_na_205ce.}

In this section we describe all distributive laws of the associative operad over itself. 
We will analyze first the versions living in the world of nonsymmetric
operads\footnote{Sometimes also called non-$\Sigma$ operads.} 
where distributive laws are given by formulas without permutation of variables, 
and then we move to the general case. 
The main result, Theorem~\ref{Zitra_Duac_v_Koline.}, states that there are only 
three non-isomorphic distributive laws --- the trivial one, the truncated one, 
and the one for nonsymmetric Poisson algebras (see Remark \ref{NSPA} below).   


\subsection{Non-$\Sigma$ version.} In this subsection we prove:

\begin{theorem}
\label{Pred_odjezdem_do_Ostravy}
The only distributive laws between two associative multiplications
that do not involve permutations of variables are given by
\[
\begin{array}{lll}
(a) &\quad (x\circ y) \bullet z \;=\; 0, &\quad x \bullet (y \circ z) \;=\; 0
\\
(b) &\quad (x\circ y) \bullet z \;=\; 0, &\quad x \bullet (y \circ z) \;=\; ( x \bullet y )\circ z
\\
(c) &\quad (x\circ y) \bullet z \;=\; x \circ( y \bullet z ), &\quad x \bullet (y \circ z) \;=\; 0
\\
(d) &\quad (x\circ y) \bullet z \;=\; x \circ( y \bullet z ), &\quad x \bullet (y \circ z) \;=\; ( x \bullet y )\circ z
\end{array}
\]
\end{theorem}

\begin{remark}
\label{NSPA}
Distributive law (a) is the trivial one. 
Distributive law (d) describes structures studied by the second author in \cite{M}, 
where they were called `nonsymmetric Poisson algebras'.
The corresponding distributive law was written as
\[
\langle x\cdot y,z\rangle = x\cdot \langle y,z\rangle, \hskip 1cm
\langle x,y\cdot z\rangle = \langle x,y\rangle \cdot z,
\] 
which is indeed a nonsymmetric form of equation \eqref{derivationlaw}. 
The same structures were later called ${\it As}^{(2)}$-algebras in \cite{encyclopedia}.
\end{remark}

\begin{proof}[Proof of Theorem~\ref{Pred_odjezdem_do_Ostravy}]
To save space, we will omit in this proof the symbol $\circ$ and write $\cdot$ instead of $\bullet$. 
We will also omit parentheses whenever the meaning is clear. 
We therefore write for example $xy \cdot z$ instead of $(x \circ y) \bullet z$.

Let $\uBB$ be the free nonsymmetric operad generated by two binary operations denoted
$xy$ and $x \cdot y$.
We use the following ordered basis for $\uBB(3)$ consisting of eight monomials:
\[
(xy)z, \quad x(yz), \quad
(x \cdot y) \cdot z, \quad x \cdot (y \cdot z), \quad
xy \cdot z, \quad x \cdot yz, \quad
(x \cdot y)z, \quad x(y \cdot z).
\]
We identify quadratic relations with row vectors of coefficients with respect to this basis.
Consider the ideal $\mathbf{I} \subset \uBB$ generated by the subspace 
$R = \mathbf{I}(3) \subset \uBB(3)$ 
which is the row space of the following matrix:
\[
[R]
=
\left[
\begin{array}{r@{\quad}r@{\qquad}r@{\quad}r@{\qquad}r@{\qquad}r@{\qquad}r@{\qquad}r}
1 & -1 & 0 &  0 & 0 & 0 & 0 & 0 \\[-1pt]
0 &  0 & 1 & -1 & 0 & 0 & 0 & 0 \\[-1pt]
0 &  0 & 0 &  0 & 1 & 0 & a & b \\[-1pt]
0 &  0 & 0 &  0 & 0 & 1 & c & d
\end{array}
\right]
\]
Row 1 expresses the associativity of $xy$.
Row 2 expresses the associativity of $x \cdot y$.
Rows 3 and 4 express two relations which may also be written as rewrite rules:
\[
\begin{array}{l@{\qquad}l@{\qquad}l}
xy \cdot z + a \; (x \cdot y)z + b \; x(y \cdot z) \equiv 0
&
\text{or}
&
xy \cdot z \;\longrightarrow\; - a \; (x \cdot y)z - b \; x(y \cdot z),
\\
x \cdot yz + c \; (x \cdot y)z + d \; x(y \cdot z) \equiv 0
&
\text{or}
&
x \cdot yz \;\longrightarrow\; - c \; (x \cdot y)z - d \; x(y \cdot z).
\end{array}
\]
These rules allow us to eliminate binary trees with root operation $x \cdot y$
by replacing them by linear combinations of binary trees\footnote{We
  use the standard bijection between monomials and rooted trees,
c.f.~Remark \ref{treeremark}.}\ with root
operation $xy$.
Let us denote the four relations corresponding to the four rows of $[R]$ as follows:
\[
\begin{array}{l}
\alpha_1(x,y,z) = (xy)z - x(yz),
\\
\alpha_2(x,y,z) = (x \cdot y) \cdot z - x \cdot (y \cdot z),
\\
\beta_1(x,y,z) = xy \cdot z + a \; (x \cdot y)z + b \; x(y \cdot z),
\\
\beta_2(x,y,z) = x \cdot yz + c \; (x \cdot y)z + d \; x(y \cdot z).
\end{array}
\] 
Let $\rho(x,y,z)$ represent any of these four relations.
Then $\rho(x,y,z)$ has ten cubic (arity 4) consequences, namely
\begin{equation}
\label{cubicconsequences}
\begin{array}{l@{\quad}l@{\quad}l@{\quad}l@{\quad}l}
\rho(wx,y,z), &
\rho(w \cdot x,y,z), &
\rho(w,xy,z), &
\rho(w,x \cdot y,z), &
\rho(w,x,yz),
\\
\rho(w,x,y \cdot z), &
\rho(w,x,y)z, &
\rho(w,x,y) \cdot z, &
w\rho(x,y,z), &
w \cdot \rho(x,y,z).
\end{array}
\end{equation}
Altogether the four relations $\alpha_1$, $\alpha_2$, $\beta_1$, $\beta_2$ have 40 cubic consequences
which span the subspace $RR = \mathbf{I}(4) \subset \uBB(4)$.
The subspace $RR$ may be identified with the row space of the $40 \times 40$ matrix $[RR]$: 
the rows correspond to the consequences of the four quadratic relations (ordered in some convenient way),
and the columns correspond to the monomial basis of $\uBB(4)$ ordered first by association type
as follows:
\begin{equation}
\label{cubicmonomials1}
\begin{array}{l}
( ( w \star_1 x ) \star_2 y ) \star_3 z,
\qquad
( w \star_1 ( x \star_2 y ) ) \star_3 z,
\qquad
( w \star_1 x ) \star_2 ( y \star_3 z ),
\\
w \star_1 ( ( x \star_2 y ) \star_3 z ),
\qquad
w \star_1 ( x \star_2 ( y \star_3 z ) ).
\end{array}
\end{equation}
Within each association type, the sequence ${\star_1}{\star_2}{\star_3}$
represents one of the eight sequences of operation symbols;
we order these as follows, where the vertical line $|$ represents the operation symbol for $xy$:
\begin{equation}
\label{cubicmonomials2}
{\star_1}{\star_2}{\star_3} \quad = \quad 
|||, \quad ||{\cdot}, \quad |{\cdot}|, \quad |{\cdot}{\cdot}, \quad 
{\cdot}||, \quad {\cdot}|{\cdot}, \quad {\cdot}{\cdot}|, \quad {\cdot}{\cdot}{\cdot},
\end{equation}
The matrix $[RR]$ has entries in the set $\{ 0, 1, -1, a, b, c, d \}$ 
and hence may be regarded as a matrix over the polynomial ring ${\bfk}[a,b,c,d]$.
This matrix is displayed in Figure \ref{matrix[RR]} with dot, $+, -$ for $0, 1, -1$ respectively. 

\begin{figure}[ht]
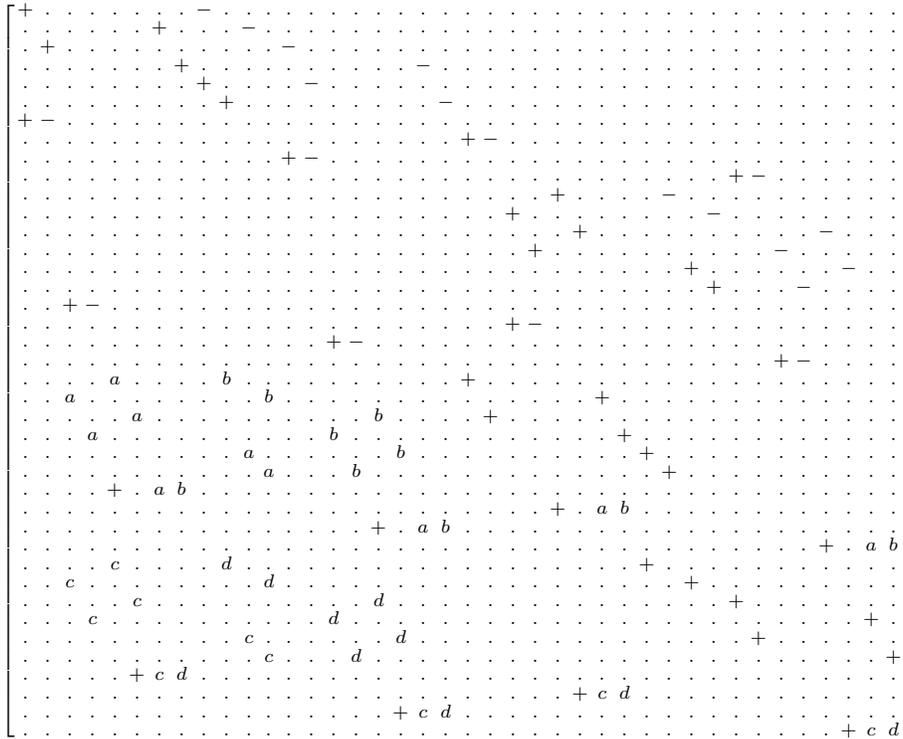

\scriptsize
$
\left[ \!\!\!
\begin{array}{
c@{\;}c@{\;}c@{\;}c@{\;}c@{\;}c@{\;}c@{\;}c@{\;}c@{\;}c@{\;}
c@{\;}c@{\;}c@{\;}c@{\;}c@{\;}c@{\;}c@{\;}c@{\;}c@{\;}c@{\;}
c@{\;}c@{\;}c@{\;}c@{\;}c@{\;}c@{\;}c@{\;}c@{\;}c@{\;}c@{\;}
c@{\;}c@{\;}c@{\;}c@{\;}c@{\;}c@{\;}c@{\;}c@{\;}c@{\;}c
}
+&.&.&.&.&.&.&.&-&.&.&.&.&.&.&.&.&.&.&.&.&.&.&.&.&.&.&.&.&.&.&.&.&.&.&.&.&.&.&.\\[-1pt] 
.&.&.&.&.&.&+&.&.&.&-&.&.&.&.&.&.&.&.&.&.&.&.&.&.&.&.&.&.&.&.&.&.&.&.&.&.&.&.&.\\[-1pt] 
.&+&.&.&.&.&.&.&.&.&.&.&-&.&.&.&.&.&.&.&.&.&.&.&.&.&.&.&.&.&.&.&.&.&.&.&.&.&.&.\\[-1pt] 
.&.&.&.&.&.&.&+&.&.&.&.&.&.&.&.&.&.&-&.&.&.&.&.&.&.&.&.&.&.&.&.&.&.&.&.&.&.&.&.\\[-1pt] 
.&.&.&.&.&.&.&.&+&.&.&.&.&-&.&.&.&.&.&.&.&.&.&.&.&.&.&.&.&.&.&.&.&.&.&.&.&.&.&.\\[-1pt] 
.&.&.&.&.&.&.&.&.&+&.&.&.&.&.&.&.&.&.&-&.&.&.&.&.&.&.&.&.&.&.&.&.&.&.&.&.&.&.&.\\[-1pt] 
+&-&.&.&.&.&.&.&.&.&.&.&.&.&.&.&.&.&.&.&.&.&.&.&.&.&.&.&.&.&.&.&.&.&.&.&.&.&.&.\\[-1pt] 
.&.&.&.&.&.&.&.&.&.&.&.&.&.&.&.&.&.&.&.&+&-&.&.&.&.&.&.&.&.&.&.&.&.&.&.&.&.&.&.\\[-1pt] 
.&.&.&.&.&.&.&.&.&.&.&.&+&-&.&.&.&.&.&.&.&.&.&.&.&.&.&.&.&.&.&.&.&.&.&.&.&.&.&.\\[-1pt] 
.&.&.&.&.&.&.&.&.&.&.&.&.&.&.&.&.&.&.&.&.&.&.&.&.&.&.&.&.&.&.&.&+&-&.&.&.&.&.&.\\[-1pt] 
.&.&.&.&.&.&.&.&.&.&.&.&.&.&.&.&.&.&.&.&.&.&.&.&+&.&.&.&.&-&.&.&.&.&.&.&.&.&.&.\\[-1pt] 
.&.&.&.&.&.&.&.&.&.&.&.&.&.&.&.&.&.&.&.&.&.&+&.&.&.&.&.&.&.&.&-&.&.&.&.&.&.&.&.\\[-1pt] 
.&.&.&.&.&.&.&.&.&.&.&.&.&.&.&.&.&.&.&.&.&.&.&.&.&+&.&.&.&.&.&.&.&.&.&.&-&.&.&.\\[-1pt] 
.&.&.&.&.&.&.&.&.&.&.&.&.&.&.&.&.&.&.&.&.&.&.&+&.&.&.&.&.&.&.&.&.&.&-&.&.&.&.&.\\[-1pt] 
.&.&.&.&.&.&.&.&.&.&.&.&.&.&.&.&.&.&.&.&.&.&.&.&.&.&.&.&.&.&+&.&.&.&.&.&.&-&.&.\\[-1pt] 
.&.&.&.&.&.&.&.&.&.&.&.&.&.&.&.&.&.&.&.&.&.&.&.&.&.&.&.&.&.&.&+&.&.&.&-&.&.&.&.\\[-1pt] 
.&.&+&-&.&.&.&.&.&.&.&.&.&.&.&.&.&.&.&.&.&.&.&.&.&.&.&.&.&.&.&.&.&.&.&.&.&.&.&.\\[-1pt] 
.&.&.&.&.&.&.&.&.&.&.&.&.&.&.&.&.&.&.&.&.&.&+&-&.&.&.&.&.&.&.&.&.&.&.&.&.&.&.&.\\[-1pt] 
.&.&.&.&.&.&.&.&.&.&.&.&.&.&+&-&.&.&.&.&.&.&.&.&.&.&.&.&.&.&.&.&.&.&.&.&.&.&.&.\\[-1pt] 
.&.&.&.&.&.&.&.&.&.&.&.&.&.&.&.&.&.&.&.&.&.&.&.&.&.&.&.&.&.&.&.&.&.&+&-&.&.&.&.\\[-1pt] 
.&.&.&.&a&.&.&.&.&b&.&.&.&.&.&.&.&.&.&.&+&.&.&.&.&.&.&.&.&.&.&.&.&.&.&.&.&.&.&.\\[-1pt] 
.&.&a&.&.&.&.&.&.&.&.&b&.&.&.&.&.&.&.&.&.&.&.&.&.&.&+&.&.&.&.&.&.&.&.&.&.&.&.&.\\[-1pt] 
.&.&.&.&.&a&.&.&.&.&.&.&.&.&.&.&b&.&.&.&.&+&.&.&.&.&.&.&.&.&.&.&.&.&.&.&.&.&.&.\\[-1pt] 
.&.&.&a&.&.&.&.&.&.&.&.&.&.&b&.&.&.&.&.&.&.&.&.&.&.&.&+&.&.&.&.&.&.&.&.&.&.&.&.\\[-1pt] 
.&.&.&.&.&.&.&.&.&.&a&.&.&.&.&.&.&b&.&.&.&.&.&.&.&.&.&.&+&.&.&.&.&.&.&.&.&.&.&.\\[-1pt] 
.&.&.&.&.&.&.&.&.&.&.&a&.&.&.&b&.&.&.&.&.&.&.&.&.&.&.&.&.&+&.&.&.&.&.&.&.&.&.&.\\[-1pt] 
.&.&.&.&+&.&a&b&.&.&.&.&.&.&.&.&.&.&.&.&.&.&.&.&.&.&.&.&.&.&.&.&.&.&.&.&.&.&.&.\\[-1pt] 
.&.&.&.&.&.&.&.&.&.&.&.&.&.&.&.&.&.&.&.&.&.&.&.&+&.&a&b&.&.&.&.&.&.&.&.&.&.&.&.\\[-1pt] 
.&.&.&.&.&.&.&.&.&.&.&.&.&.&.&.&+&.&a&b&.&.&.&.&.&.&.&.&.&.&.&.&.&.&.&.&.&.&.&.\\[-1pt] 
.&.&.&.&.&.&.&.&.&.&.&.&.&.&.&.&.&.&.&.&.&.&.&.&.&.&.&.&.&.&.&.&.&.&.&.&+&.&a&b\\[-1pt] 
.&.&.&.&c&.&.&.&.&d&.&.&.&.&.&.&.&.&.&.&.&.&.&.&.&.&.&.&+&.&.&.&.&.&.&.&.&.&.&.\\[-1pt] 
.&.&c&.&.&.&.&.&.&.&.&d&.&.&.&.&.&.&.&.&.&.&.&.&.&.&.&.&.&.&+&.&.&.&.&.&.&.&.&.\\[-1pt] 
.&.&.&.&.&c&.&.&.&.&.&.&.&.&.&.&d&.&.&.&.&.&.&.&.&.&.&.&.&.&.&.&+&.&.&.&.&.&.&.\\[-1pt] 
.&.&.&c&.&.&.&.&.&.&.&.&.&.&d&.&.&.&.&.&.&.&.&.&.&.&.&.&.&.&.&.&.&.&.&.&.&.&+&.\\[-1pt] 
.&.&.&.&.&.&.&.&.&.&c&.&.&.&.&.&.&d&.&.&.&.&.&.&.&.&.&.&.&.&.&.&.&+&.&.&.&.&.&.\\[-1pt] 
.&.&.&.&.&.&.&.&.&.&.&c&.&.&.&d&.&.&.&.&.&.&.&.&.&.&.&.&.&.&.&.&.&.&.&.&.&.&.&+\\[-1pt] 
.&.&.&.&.&+&c&d&.&.&.&.&.&.&.&.&.&.&.&.&.&.&.&.&.&.&.&.&.&.&.&.&.&.&.&.&.&.&.&.\\[-1pt] 
.&.&.&.&.&.&.&.&.&.&.&.&.&.&.&.&.&.&.&.&.&.&.&.&.&+&c&d&.&.&.&.&.&.&.&.&.&.&.&.\\[-1pt] 
.&.&.&.&.&.&.&.&.&.&.&.&.&.&.&.&.&+&c&d&.&.&.&.&.&.&.&.&.&.&.&.&.&.&.&.&.&.&.&.\\[-1pt] 
.&.&.&.&.&.&.&.&.&.&.&.&.&.&.&.&.&.&.&.&.&.&.&.&.&.&.&.&.&.&.&.&.&.&.&.&.&+&c&d 
\end{array}
\!\!\! \right]
$
\vspace{-3mm}
\caption{Matrix $[RR]$: cubic consequences of quadratic relations}
\label{matrix[RR]}
\end{figure}

To understand how the rank of $[RR]$ depends on the parameters $a, b, c, d$ we first use elementary
row and column operations to compute a partial Smith form as described in 
\cite[Chapter 8]{BD-book}.
Roughly speaking, we repeatedly move entries equal to $\pm 1$ to the upper left diagonal of the
matrix, change their signs if necessary, and then use each resulting diagonal 1 to eliminate the
entries below and to the right, continuing until the lower right block no longer contains a nonzero
scalar.
When this computation terminates, we have reduced $[RR]$ to the block-diagonal matrix 
$\mathrm{diag}(I_{32},L)$,
which is row-column equivalent to $[RR]$ and hence has the same rank as $[RR]$,
where $L$ is an $8 \times 8$ matrix over ${\bfk}[a,b,c,d]$ which has two zero rows and two zero columns.
After deleting these superfluous rows and columns, we obtain this $6 \times 6$ matrix:
\[
L' 
=
\left[
\begin{array}{cccccc}
-a d  &  -b^2-b  &  a^2-a c  &  0  &  0  &  0
\\[-1pt]
0  &  b d+d  &  -a c-a  &  0  &  0  &  0
\\[-1pt]
a d  &  b d-d^2  &  c^2+c  &  0  &  0  &  0
\\[-1pt]
0  &  0  &  0  &  -b^2-b  &  -a b-a  &  -a^2-a b
\\[-1pt]
0  &  0  &  0  &  -a d  &  0  &  a d
\\[-1pt]
0  &  0  &  0  &  -c d-d^2  &  -c d-d  &  -c^2-c
\end{array}
\right]
\]
In order for the map representing the distributive law to be an isomorphism, 
it is necessary and sufficient
that $[RR]$ have rank 32, or equivalently that $L'$ be the zero matrix.
Consider the set $G$ consisting of the nonzero entries of $L'$.
We compute a Gr\"obner basis for the ideal $J \subset {\bfk}[a,b,c,d]$ 
generated by $G$ with respect 
to the deglex monomial order determined by $a \prec b \prec c \prec d$.
This Gr\"obner basis for $J$ consists of the polynomials $a$, $d$, $b(b{+}1)$, $c(c{+}1)$.
Hence $J$ is a zero-dimensional ideal whose zero set consists of exactly four points:
\[
(a,b,c,d) \;\; = \;\; (0,0,0,0), \quad (0,0,-1,0), \quad (0,-1,0,0), \quad (0,-1,-1,0).
\]
These solutions correspond to the following pairs of rewrite rules
\[
\begin{array}{lll}
(a) &\quad xy \cdot z \;\longrightarrow\; 0, &\quad x \cdot yz \;\longrightarrow\; 0
\\
(b) &\quad xy \cdot z \;\longrightarrow\; 0, &\quad x \cdot yz \;\longrightarrow\; ( x \cdot y ) z
\\
(c) &\quad xy \cdot z \;\longrightarrow\; x ( y \cdot z ), &\quad x \cdot yz \;\longrightarrow\; 0
\\
(d) &\quad xy \cdot z \;\longrightarrow\; x ( y \cdot z ), &\quad x \cdot yz \;\longrightarrow\; ( x \cdot y ) z
\end{array}
\]
which give the four nonsymmetric
laws $\dl{\uAss}{\uAss}$  of Theorem~\ref{Pred_odjezdem_do_Ostravy}.
\end{proof}


\subsection{General version}

In this subsection we generalize Theorem \ref{Pred_odjezdem_do_Ostravy} by allowing permutations 
of variables:

\begin{theorem}
\label{vcera_jsem_se_proti_vetru_nadrel}
The only distributive laws between two associative multiplications
are the four laws of Theorem \ref{Pred_odjezdem_do_Ostravy} together with the following three:
\[
\begin{array}{lll}
(e) 
&\qquad 
(x \circ y) \bullet z \,=\, 0, 
&\qquad
x \bullet (y \circ z) \,=\, y \circ ( x \bullet z )
\\
(f) 
&\qquad 
(x \circ y) \bullet z \,=\, ( x \bullet z )  \circ y,
&\qquad
x \bullet (y \circ z) \,=\, 0
\\
(g) 
&\qquad 
(x \circ y) \bullet z \,=\, ( x \bullet z ) \circ y,
&\qquad
x \bullet (y \circ z) \,=\, y \circ ( x \bullet z ).
\end{array}
\]
\end{theorem}

The proof is postponed to the end of this subsection. 
We note that the rewrite rule 
$(x \circ y) \bullet z \,=\, ( x \bullet z ) \circ y$ states that 
the right multiplications $- \circ y$ and $- \bullet z$ commute; 
similarly, $x \bullet (y \circ z) \,=\, y  \circ( x \bullet z )$ states that the left multiplications 
$y \circ -$ and $x \bullet -$ commute.

Let us denote by $\oA_a, \ldots \oA_g$ the operads defined by
distributive laws (a)--(g) of Theorems~\ref{Pred_odjezdem_do_Ostravy}
and~\ref{vcera_jsem_se_proti_vetru_nadrel} (in the given order). It turns
out that these operads fall into three isomorphism classes:
$\{\oA_a\}$,
$\{\oA_b,\oA_c,\oA_e,\oA_f\}$, and
$\{\oA_d,\oA_g\}$. 
The corresponding isomorphisms are given by changing one or both multiplications into the opposite, 
that is $\circ \mapsto \circ^\op$ and/or $\bullet \mapsto  \bullet^\op$. 
It is easy to verify that one gets the following isomorphism diagrams:
\[
\adjustbox{valign=m}{
\xymatrix@C=6em@R=4em{\oA_b\ar[r]_\cong^{\circ \mapsto \circ^\op}\ar[d]_{\bullet
\mapsto  \bullet^\op}^\cong & \oA_e\ar[d]_\cong^{\bullet
\mapsto  \bullet^\op}
\\
\oA_f\ar[r]^{\circ \mapsto \circ^\op}_\cong & \oA_c
}
}
\hskip 1em \hbox { and } \hskip 1em
\adjustbox{valign=m}{
\xymatrix@C=6em@R=4em{\oA_d\ar[r]^{\circ \mapsto \circ^\op}_\cong\ar[d]_{\bullet
\mapsto  \bullet^\op}^\cong & \oA_g\ar[d]^{\bullet
\mapsto  \bullet^\op}_\cong
\\
\oA_g\ar[r]^{\circ \mapsto \circ^\op}_\cong &\ \oA_d .
}
}
\]
One therefore has:

\begin{theorem}
\label{Zitra_Duac_v_Koline.}
There are precisely three nonisomorphic distributive laws between two associative
multiplications, namely
\begin{itemize}[leftmargin=*]
\item 
the trivial law (a),
\item
the truncated law represented by rewrite rules (b), (c), (e) or (f),
and 
\item
the law for nonsymmetric Poisson algebras represented by (d) or (g).
\end{itemize} 
\end{theorem}

\begin{remark}
Note that
the operads $\oA_a$, $\oA_b$, $\oA_c$, and $\oA_d$ defined by
distributive laws (a)--(d) of Theorem~\ref{Pred_odjezdem_do_Ostravy}
are mutually non-isomorphic in the category of non-$\Sigma$
operads. Therefore in the category of algebras over non-symmetric
operads there are four different distributive laws between two
associative multiplications.
\end{remark}

Theorem \ref{Zitra_Duac_v_Koline.} has the 
following simple but very interesting consequence:

\begin{corollary}
\label{I_v_susarne_je_strasne_vedro.}
Up to isomorphism, 
the only distributive law between two associative multiplications
in the monoidal category of sets
is that of nonsymmetric Poisson algebras.
\end{corollary}

\begin{example}
\label{byhandexample}
Let us verify `by hand' that~(e) indeed determines a distributive law. 
We must verify that it is compatible with the associativity of $\bullet$ and $\circ$. 
We also need to check that the result of repeated applications of (e) does not depend on their order. 
Theorem \ref{turmo} tells us that it suffices to consider only expressions involving four variables.

{\it Compatibility with the associativity of $\circ$}. 
The associativity of $\circ$ means that
\[
\big(( y \circ z) \circ w\big) = \big( y \circ (z \circ w)\big),  
\] 
for arbitrary symbols $y,z,w$. Thus, for a symbol
$x$, one has
\begin{equation}
\label{streda_v_Ostrave}
x \bullet   \big(( y \circ z) \circ w\big) = 
x \bullet  \big( y \circ (z \circ w)\big)  .
\end{equation}
The compatibility with associativity means that both sides of this
equation remain equal after we apply, possibly repeatedly, rule (e) to
them. For the left side of~(\ref{streda_v_Ostrave}) we get
\[
x \bullet   \big(( y \circ z) \circ w\big) = (y \circ z) \circ (x
\bullet w),
\]
while the right hand side becomes
\[
x \bullet  \big( y \circ (z \circ w)\big)  =
y \circ   \big( x \bullet (z \circ w)\big)  =
y \circ  \big( z \circ (x \bullet w)\big).
\] 
So we need to check whether
\[
(y \circ z) \circ (x \bullet w) =  y \circ  \big( z \circ (x \bullet w)\big).
\]
This equality follows from the associativity of
$\circ$~. We need to do the same analysis for
\[
\big ( (x \circ   y) \circ z) \bullet w = 
\big( x \circ  ( y \circ z)\big) \bullet w.
\]
In this case~(e) turns both sides into $0$.

{\it Compatibility with the associativity of~$\bullet$}. 
We need to consider three equations implied by the associativity 
of~$\bullet$. The first one is
\[
(x \bullet y) \bullet (z \circ w) = x \bullet \big( y \bullet (z \circ w)\big). 
\]
Modifying the left hand side using~(e) gives
\[
(x \bullet y) \bullet (z \circ w) = 
z \circ \big(( x \bullet y)\bullet w),
\]
while
\[
x \bullet \big( y \bullet (z \circ w)\big) = x \bullet \big( z \circ
(y \bullet w) \big) = z \circ \big( x \bullet (y \bullet w)\big).
\]
However thanks to the associativity of~$\bullet$ we have
\[
z \circ (\big( x \bullet y)\bullet w) =  z \circ \big( x \bullet (y
\bullet w)\big).
\]
The next equation to analyze is
\[
\big(x \bullet (y \circ z)\big) \bullet w = 
x \bullet\big( (y \circ z) \bullet w \big).
\]
The left side expands as
\[
\big(x \bullet (y \circ z)\big) \bullet w = 
\big(y \circ (x \bullet z)\big) \bullet w = 0, 
\]
while the right side is seen to be zero immediately. 
The last equation to be considered is
\[
\big ((x \circ y)\bullet z) \bullet w = (x \circ y) \bullet (z \bullet w).
\]
But applying~(e) turns both sides immediately to zero.

{\it Independence of order}. 
All expressions featured above offered at most one way to
apply~(e). This is not true for
\[
(x \circ y) \bullet (z \circ w).
\]
Applying the first rule of~(e) first, with $(z \circ w)$ instead of
$z$, turns it into zero, while
applying the second rule of~(e) first we get
\[
(x \circ y) \bullet (z \circ w) = z \circ \big((x \circ y) \bullet w)\big),
\]
which is zero again, by the first rule of~(e). 
It is not difficult to see that the above finite number of cases was
all we needed to check, thus the
verification that (e) defines a~distributive law is finished.

\def\rightl{
\qbezier(1.5,-1.5)(5,-5)(8,-8)
}
\def\leftl{
\qbezier(-2,-2)(-5,-5)(-8,-8)
}

\def\Atree#1#2#3#4#5{
\unitlength=1pt
\begin{picture}(20.00,20.00)(0.00,0.00)
\thicklines
\put(0,0){
\put(0,0){\makebox(0.00,0.00){$#2$}}
\put(-10.00,10){\makebox(0.00,0.00){$#1$}}
\put(0,0){\leftl}
\put(-10,10){\rightl}
\put(0,0){\rightl}
\put(-12,9){\line(-1,-1){18.00}}
\put(-10,12.00){\vector(0,1){15}}
\put(-30,-15){\makebox(0.00,0.00)[t]{$#3$}}
\put(-10,-15){\makebox(0.00,0.00)[t]{$#4$}}
\put(10,-15){\makebox(0.00,0.00)[t]{$#5$}}
}
\end{picture}
}

\def\Btree#1#2#3#4#5{
\unitlength=1pt
\begin{picture}(20.00,20.00)(0,0)
\thicklines
\put(0,0){
\put(0,0){\makebox(0.00,0.00){$#2$}}
\put(10.00,10){\makebox(0.00,0.00){$#1$}}
\put(0,0){\rightl}
\put(10,10){\leftl}
\put(0,0){\leftl}
\put(12,9){\line(1,-1){18.00}}
\put(10,12.00){\vector(0,1){15.00}}
\put(30,-15){\makebox(0.00,0.00)[t]{$#5$}}
\put(10,-15){\makebox(0.00,0.00)[t]{$#4$}}
\put(-10,-15){\makebox(0.00,0.00)[t]{$#3$}}
}
\end{picture}
}

\begin{remark}
\label{treeremark}
The above calculations can be visualized by labelled
planar rooted trees. 
Representing the $\circ$-multiplication by a white vertex with two
inputs and one output, and the $\bullet$-multiplication by a similar black
vertex,
the associativity of $\circ$ and $\bullet$ can be depicted~as
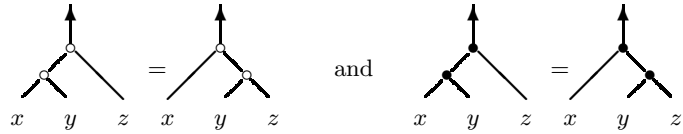
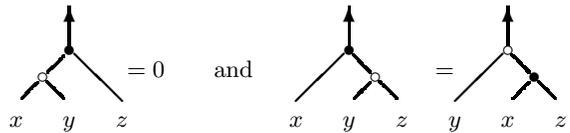
\[
\raisebox{-2em}{}
\raisebox{3em}{}
\Btree\circ\circ xyz \hskip 1.8em =  \hskip 3em \Atree\circ\circ xyz \hskip -.6em 
\qquad \mbox {and} \qquad \hskip 1em 
\Btree\bullet\bullet xyz  \hskip 1.8em =  \hskip 3em \Atree\bullet\bullet xyz \hskip -.6em 
\]
while  rule~(e) reads
\[
\raisebox{-2em}{}
\raisebox{3em}{}
\Btree\bullet\circ xyz \hskip 1em = 0 \qquad \mbox {and} \qquad \hskip 3em 
\Atree\bullet\circ xyz  =  \hskip 3em \Atree\circ\bullet yxz \hskip -.6em .
\]
A pictorial verification of the compatibility of rule~(e) with equation \eqref{streda_v_Ostrave}
is shown in Figure~\ref{po_navratu_z_Ostravy}; 
the remaining (and in fact easier) cases can be verified similarly.
\end{remark}

\begin{figure}[ht]
\begin{center}
\unitlength=.8pt
\thicklines
\begin{picture}(60.00,240.00)(20.00,0.00)
\put(-55,180){
\put(30.00,20.00){\makebox(0.00,0.00){$\circ$}}
\put(60.00,03.00){\makebox(0.00,0.00)[t]{$w$}}
\put(40.00,3.00){\makebox(0.00,0.00)[t]{$z$}}
\put(20.00,3.00){\makebox(0.00,0.00)[t]{$y$}}
\put(0.00,3.00){\makebox(0.00,0.00)[t]{$x$}}
\put(40.00,30.00){\makebox(0.00,0.00){$\circ$}}
\put(30.00,40.00){\makebox(0.00,0.00){$\bullet$}}
\put(30.00,42.00){\vector(0,1){15}}
\put(42.00,28.00){\line(1,-1){18.00}}
\put(28.00,38.00){\line(-1,-1){28}}
\put(30.00,40.00){\rightl}
\put(30.00,20.00){\rightl}
\put(30.00,20.00){\leftl}
\put(40.00,30.00){\leftl}
}
\put(30,210.00){\makebox(0.00,0.00){$=$}}
\put(55,180){
\put(60.00,03.00){\makebox(0.00,0.00)[t]{$w$}}
\put(40.00,3.00){\makebox(0.00,0.00)[t]{$z$}}
\put(20.00,3.00){\makebox(0.00,0.00)[t]{$y$}}
\put(0.00,3.00){\makebox(0.00,0.00)[t]{$x$}}
\put(50.00,20.00){\makebox(0.00,0.00){$\circ$}}
\put(40.00,30.00){\makebox(0.00,0.00){$\circ$}}
\put(30.00,40.00){\makebox(0.00,0.00){$\bullet$}}
\put(30.00,42.00){\vector(0,1){15}}
\put(50.00,20.00){\leftl}
\put(38.00,28.00){\line(-1,-1){18.00}}
\multiput(30,40)(10,-10){3}{\put(0,0){\rightl}}
\put(28.00,38.00){\line(-1,-1){28.00}}
}
\put(110,90){
\put(60.00,03.00){\makebox(0.00,0.00)[t]{$w$}}
\put(40.00,3.00){\makebox(0.00,0.00)[t]{$z$}}
\put(20.00,3.00){\makebox(0.00,0.00)[t]{$x$}}
\put(0.00,3.00){\makebox(0.00,0.00)[t]{$y$}}
\put(50.00,20.00){\makebox(0.00,0.00){$\circ$}}
\put(40.00,30.00){\makebox(0.00,0.00){$\bullet$}}
\put(30.00,40.00){\makebox(0.00,0.00){$\circ$}}
\put(30.00,42.00){\vector(0,1){15}}
\put(50.00,20.00){\leftl}
\put(38.00,28.00){\line(-1,-1){18.00}}
\multiput(30,40)(10,-10){3}{\put(0,0){\rightl}}
\put(28.00,38.00){\line(-1,-1){28.00}}
}
\put(55,00){
\put(60.00,03.00){\makebox(0.00,0.00)[t]{$w$}}
\put(40.00,3.00){\makebox(0.00,0.00)[t]{$x$}}
\put(20.00,3.00){\makebox(0.00,0.00)[t]{$z$}}
\put(0.00,3.00){\makebox(0.00,0.00)[t]{$y$}}
\put(50.00,20.00){\makebox(0.00,0.00){$\bullet$}}
\put(40.00,30.00){\makebox(0.00,0.00){$\circ$}}
\put(30.00,40.00){\makebox(0.00,0.00){$\circ$}}
\put(30.00,42.00){\vector(0,1){15}}
\put(50.00,20.00){\leftl}
\put(38.00,28.00){\line(-1,-1){18.00}}
\multiput(30,40)(10,-10){3}{\put(0,0){\rightl}}
\put(28.00,38.00){\line(-1,-1){28.00}}
}
\put(-55,0){
\put(60.00,03.00){\makebox(0.00,0.00)[t]{$w$}}
\put(40.00,3.00){\makebox(0.00,0.00)[t]{$x$}}
\put(20.00,3.00){\makebox(0.00,0.00)[t]{$z$}}
\put(0.00,3.00){\makebox(0.00,0.00)[t]{$y$}}
\put(30.00,40.00){\makebox(0.00,0.00){$\circ$}}
\put(50.00,20.00){\makebox(0.00,0.00){$\bullet$}}
\put(50.00,20.00){\leftl}
\put(50.00,20.00){\rightl}
\put(10.00,20.00){\leftl}
\put(10.00,20.00){\rightl}
\put(10.00,20.00){\makebox(0.00,0.00){$\circ$}}
\put(30.00,42.00){\vector(0,1){15}}
\put(32.00,38.00){\line(1,-1){16.00}}
\put(28.00,38.00){\line(-1,-1){16.00}}
}
\put(30,40){\makebox(0.00,0.00){$=$}}
\put(115,70){\makebox(0.00,0.00){\rotatebox[origin=c]{-45}{$\Big\Downarrow$}}}
\put(115,160){\makebox(0.00,0.00){\rotatebox[origin=c]{45}{$\Big\Downarrow$}}}
\put(-25,120){\makebox(0.00,0.00){$\Big\Downarrow$}}
\end{picture} 
\end{center}
\caption{Tree diagrams for compatibility proof}
\label{po_navratu_z_Ostravy}
\end{figure}
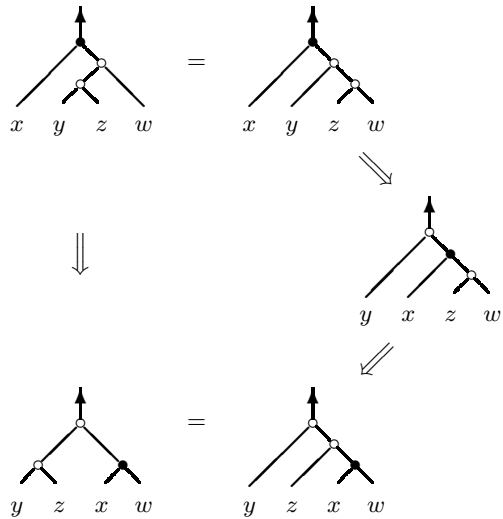
\end{example}

\begin{proof}[Proof of Theorem \ref{vcera_jsem_se_proti_vetru_nadrel}]
We use the same conventions regarding the notation for the $\circ$ and
$\bullet$ products as in the proof of Theorem
\ref{Pred_odjezdem_do_Ostravy}.  
The method for the symmetric case is essentially the same as for the nonsymmetric case, although
the matrices and the number of parameters are six times larger.
Let $\BB$ be the free symmetric operad generated by two binary operations denoted
$xy$ and $x \cdot y$.
We use the following ordered basis for $\BB(3)$ consisting of 48 monomials: 
\[
\begin{array}{l@{\qquad}l@{\qquad}l@{\qquad}l}
( x^\sigma y^\sigma ) z^\sigma, & 
x^\sigma ( y^\sigma z^\sigma ), &
( x^\sigma \cdot y^\sigma ) \cdot z^\sigma, & 
x^\sigma \cdot ( y^\sigma \cdot z^\sigma ), 
\\
x^\sigma y^\sigma \cdot z^\sigma, & 
x^\sigma \cdot y^\sigma z^\sigma, &
( x^\sigma \cdot y^\sigma ) z^\sigma, & 
x^\sigma ( y^\sigma \cdot z^\sigma ).
\end{array}
\]
The permutations $\sigma \in S_3$ permuting the arguments $x, y, z$ (not the positions) are in 
lexicographical order.
We identify quadratic relations with row vectors of coefficients with respect to this basis.
Consider the ideal $\mathbf{I} \subset \BB$ generated by the subspace 
$R = \mathbf{I}(3)$ which is the row space of the following block matrix:
\begin{equation}
\label{quadraticmatrix1}
[R]
=
\left[
\begin{array}{cccccccc}
I_6   & -I_6  & \cdot & \cdot & \cdot & \cdot & \cdot & \cdot \\
\cdot & \cdot &   I_6 &  -I_6 & \cdot & \cdot & \cdot & \cdot \\
\cdot & \cdot & \cdot & \cdot &   I_6 & \cdot &     A &     B \\
\cdot & \cdot & \cdot & \cdot & \cdot &   I_6 &     C &     D 
\end{array}
\right]
\end{equation}
We write $I_6$ and dot for the $6 \times 6$ identity and zero matrices, together with
\begin{equation}
\label{quadraticmatrix2}
A 
= 
\left[
\begin{array}{cccccc}
a_1 & a_2 & a_3 & a_4 & a_5 & a_6 \\[-3pt]
a_2 & a_1 & a_5 & a_6 & a_3 & a_4 \\[-3pt]
a_3 & a_4 & a_1 & a_2 & a_6 & a_5 \\[-3pt]
a_5 & a_6 & a_2 & a_1 & a_4 & a_3 \\[-3pt]
a_4 & a_3 & a_6 & a_5 & a_1 & a_2 \\[-3pt]
a_6 & a_5 & a_4 & a_3 & a_2 & a_1 
\end{array}
\right],
\end{equation}
and similarly for $B$, $C$ and $D$.
Thus $[R]$ contains 24 parameters.
We point out that rows 1, 7, 13, 19 generate the row space of $[R]$ as an $S_3$-module:
rows 1 and 7 represent associativity for operations $xy$ and $x \cdot y$; rows 13 and 19
represent the rewrite rules which show how to express a binary tree with operation $x \cdot y$ 
at the root as a linear combination of binary trees with operation $xy$ at the root:
\[
\begin{array}{l}
x y \cdot z +
a_1 ( x \cdot y ) z +
a_2 ( x \cdot z ) y +
a_3 ( y \cdot x ) z +
a_4 ( y \cdot z ) x +
a_5 ( z \cdot x ) y +
a_6 ( z \cdot y ) x 
\\[2pt]
\quad
{} +
b_1 x ( y \cdot z ) +
b_2 x ( z \cdot y ) +
b_3 y ( x \cdot z ) +
b_4 y ( z \cdot x ) +
b_5 z ( x \cdot y ) +
b_6 z ( y \cdot x )
\equiv 0,
\\[3pt]
x \cdot y z +
c_1 ( x \cdot y ) z +
c_2 ( x \cdot z ) y +
c_3 ( y \cdot x ) z +
c_4 ( y \cdot z ) x +
c_5 ( z \cdot x ) y +
c_6 ( z \cdot y ) x 
\\[2pt]
\quad
{} +
d_1 x ( y \cdot z ) +
d_2 x ( z \cdot y ) +
d_3 y ( x \cdot z ) +
d_4 y ( z \cdot x ) +
d_5 z ( x \cdot y ) +
d_6 z ( y \cdot x )
\equiv 0.
\end{array}
\]
Let $\rho(x,y,z)$ be the relation represented by one of the rows 1, 7, 13, 19.
Each of these four relations has ten cubic consequences as in equation \eqref{cubicconsequences},
for a total of 40 relations which generate the $S_4$-module 
$RR = \mathbf{I}(4) \subset \BB(4)$. 
Each of these 40 relations has 24 permutations, for a total of 960 relations which span $RR$
as a subspace of $\BB(4)$.
If we apply the 24 permutations of $w, x, y, z$ to the 40 nonsymmetric monomials in
equations \eqref{cubicmonomials1}-\eqref{cubicmonomials2} 
then we obtain 960 monomials which form an ordered basis 
of $\BB(4)$.
Thus we can represent $RR$ as the row space of a $960 \times 960$ matrix $[RR]$ whose entries
belong to
\[
\{ 0, \pm 1 \} \cup X, 
\qquad \text{where} \qquad
X = \{ a_k, b_k, c_k, d_k \mid 1 \le k \le 6 \}.
\]
Thus $[RR]$ may be regarded as a matrix over
the polynomial ring ${\bfk}[X]$ with 24 variables.
As in the nonsymmetric case, we compute a partial Smith form for $[RR]$ and 
obtain a~block diagonal matrix $\mathrm{diag}( I_{768}, L )$ 
where $L$ has size $192 \times 192$ and contains no nonzero scalar entries.
The set of nonzero entries of $L$ contains 575 polynomials, all of which have total degree 1 or 2
in the variables $X$.
From this large set of ideal generators we obtain a~deglex Gr\"obner basis of only 28 polynomials:
\[
\begin{array}{l}
a_1, \;\; 
a_3, \;\; 
a_4, \;\; 
a_5, \;\; 
a_6, \;\; 
b_2, \;\; 
b_3, \;\; 
b_4, \;\; 
b_5, \;\; 
b_6, \;\; 
c_2, \;\; 
c_3, \;\; 
c_4, \;\; 
c_5, \;\; 
c_6, \;\; 
d_1, \;\; 
d_2, \;\; 
d_4, \;\; 
d_5, \;\; 
d_6, \;\; 
\\[2pt]
a_2^2 + a_2, \;\; 
a_2 b_1, \;\; 
a_2 c_1, \;\; 
b_1^2 + b_1, \;\; 
b_1 d_3, \;\; 
c_1^2 + c_1, \;\; 
c_1 d_3, \;\; 
d_3^2 + d_3.
\end{array}
\]
From this we easily determine that the ideal is zero-dimensional and that its zero set consists of
the following seven points:
\[
\begin{array}{
r@{\,\,}|
r@{\,}r@{\,}r@{\,}r@{\,}r@{\,}r@{\,\,}|@{\,}r@{\,}r@{\,}r@{\,}r@{\,}r@{\,}r@{\,\,}|@{\,}
r@{\,}r@{\,}r@{\,}r@{\,}r@{\,}r@{\,\,}|@{\,}r@{\,}r@{\,}r@{\,}r@{\,}r@{\,}r@{\,\,}|
}
&
a_1 & a_2 & a_3 & a_4 & a_5 & a_6 & b_1 & b_2 & b_3 & b_4 & b_5 & b_6 & 
c_1 & c_2 & c_3 & c_4 & c_5 & c_6 & d_1 & d_2 & d_3 & d_4 & d_5 & d_6 \\
\midrule
1 &
\cdot &  \cdot & \cdot & \cdot & \cdot & \cdot &  \cdot & \cdot & \cdot & \cdot & \cdot & \cdot &  \cdot & \cdot & \cdot & \cdot & \cdot & \cdot & \cdot & \cdot &  \cdot & \cdot & \cdot & \cdot \\[-2pt]
2 &
\cdot &  \cdot & \cdot & \cdot & \cdot & \cdot &  \cdot & \cdot & \cdot & \cdot & \cdot & \cdot & -1 & \cdot & \cdot & \cdot & \cdot & \cdot & \cdot & \cdot &  \cdot & \cdot & \cdot & \cdot \\[-2pt]
3 &
\cdot &  \cdot & \cdot & \cdot & \cdot & \cdot & -1 & \cdot & \cdot & \cdot & \cdot & \cdot &  \cdot & \cdot & \cdot & \cdot & \cdot & \cdot & \cdot & \cdot &  \cdot & \cdot & \cdot & \cdot \\[-2pt]
4 &
\cdot &  \cdot & \cdot & \cdot & \cdot & \cdot & -1 & \cdot & \cdot & \cdot & \cdot & \cdot & -1 & \cdot & \cdot & \cdot & \cdot & \cdot & \cdot & \cdot &  \cdot & \cdot & \cdot & \cdot \\
\midrule
5 &
\cdot &  \cdot & \cdot & \cdot & \cdot & \cdot &  \cdot & \cdot & \cdot & \cdot & \cdot & \cdot &  \cdot & \cdot & \cdot & \cdot & \cdot & \cdot & \cdot & \cdot & -1 & \cdot & \cdot & \cdot \\[-2pt]
6 &
\cdot & -1 & \cdot & \cdot & \cdot & \cdot &  \cdot & \cdot & \cdot & \cdot & \cdot & \cdot &  \cdot & \cdot & \cdot & \cdot & \cdot & \cdot & \cdot & \cdot &  \cdot & \cdot & \cdot & \cdot \\[-2pt]
7 &
\cdot & -1 & \cdot & \cdot & \cdot & \cdot &  \cdot & \cdot & \cdot & \cdot & \cdot & \cdot &  \cdot & \cdot & \cdot & \cdot & \cdot & \cdot & \cdot & \cdot & -1 & \cdot & \cdot & \cdot \\
\midrule
\end{array}
\]
For points 1--4, the matrices $A, B, C, D$ from equations 
\eqref{quadraticmatrix1}-\eqref{quadraticmatrix2}
are as follows:
\[
\left[
\begin{array}{cc}
A & B \\
C & D 
\end{array}
\right]
=
\left[
\begin{array}{cc}
0 & 0 \\
0 & 0 
\end{array}
\right],
\;
\left[
\begin{array}{cc}
   0 & 0 \\
-I_6 & 0 
\end{array}
\right],
\;
\left[
\begin{array}{cc}
0 & -I_6 \\
0 &    0 
\end{array}
\right],
\;
\left[
\begin{array}{cc}
   0 & -I_6 \\
-I_6 &    0 
\end{array}
\right].
\]
The corresponding distributive laws are simply the symmetrizations of the four laws from 
the nonsymmetric case.
Points (5)-(7) give new symmetric distributive laws which have no analogue in the nonsymmetric case.
Consider these (negative) permutation~\hbox{matrices:}
\[
P = 
\left[ 
\begin{array}{r@{\;\;}r@{\;\;}r@{\;\;}r@{\;\;}r@{\;\;}r}
 \cdot & -1 &  \cdot &  \cdot &  \cdot &  \cdot \\[-2pt]
-1 &  \cdot &  \cdot &  \cdot &  \cdot &  \cdot \\[-2pt]
 \cdot &  \cdot &  \cdot & -1 &  \cdot &  \cdot \\[-2pt]
 \cdot &  \cdot & -1 &  \cdot &  \cdot &  \cdot \\[-2pt]
 \cdot &  \cdot &  \cdot &  \cdot &  \cdot & -1 \\[-2pt]
 \cdot &  \cdot &  \cdot &  \cdot & -1 &  \cdot 
\end{array}
\right],
\qquad
Q = 
\left[ 
\begin{array}{r@{\;\;}r@{\;\;}r@{\;\;}r@{\;\;}r@{\;\;}r}
 \cdot &  \cdot & -1 &  \cdot &  \cdot &  \cdot \\[-2pt]
 \cdot &  \cdot &  \cdot &  \cdot & -1 &  \cdot \\[-2pt]
-1 &  \cdot &  \cdot &  \cdot &  \cdot &  \cdot \\[-2pt]
 \cdot &  \cdot &  \cdot &  \cdot &  \cdot & -1 \\[-2pt]
 \cdot & -1 &  \cdot &  \cdot &  \cdot &  \cdot \\[-2pt]
 \cdot &  \cdot &  \cdot & -1 &  \cdot &  \cdot 
\end{array}
\right].
\]
Then points 5--7 correspond to 
\[
\left[
\begin{array}{cc}
A & B \\
C & D 
\end{array}
\right]
=
\left[
\begin{array}{cc}
0 & 0 \\
0 & Q 
\end{array}
\right],
\;
\left[
\begin{array}{cc}
P & 0 \\
0 & 0 
\end{array}
\right],
\;
\left[
\begin{array}{cc}
P & 0 \\
0 & Q 
\end{array}
\right].
\]
These solutions correspond respectively to (all permutations of) 
these rewrite rules:
\[
\begin{array}{lll}
5\colon
&\quad 
xy \cdot z \,\longrightarrow\, 0, 
&\quad
x \cdot yz \,\longrightarrow\, y ( x \cdot z )
\\
6\colon 
&\quad 
xy \cdot z \,\longrightarrow\, ( x \cdot z ) y,
&\quad
x \cdot yz \,\longrightarrow\, 0
\\
7\colon
&\quad 
xy \cdot z \,\longrightarrow\, ( x \cdot z ) y,
&\quad
x \cdot yz \,\longrightarrow\, y ( x \cdot z ).
\end{array}
\]
These are the three remaining distributive laws of Theorem~\ref{vcera_jsem_se_proti_vetru_nadrel}.
\end{proof}


\section{Distributive laws $\dl\Com\Ass$}

\begin{theorem}
The only distributive law  $\dl\Com\Ass$ is the trivial one.
\end{theorem}

\begin{proof}
We write $ab$ for the associative operation, and $a \cdot b$ for the
commutative associative operation.
Commutativity implies that we need to consider only six association types in arity 3,
which we order as follows:
\[
{\ast} \cdot {\ast} \cdot {\ast} = ( {\ast} \cdot {\ast} ) \cdot {\ast}, 
\qquad 
( {\ast} {\ast} ) \cdot {\ast},  
\qquad 
( {\ast} \cdot {\ast} ) {\ast},  
\qquad 
{\ast}{\ast}{\ast} = ( {\ast} {\ast} ) {\ast},  
\qquad 
{\ast} ( {\ast} \cdot {\ast} ), 
\qquad 
{\ast} ( {\ast} {\ast} ). 
\]
Similarly, we need consider only 25 association types in arity 4;
in the following ordered list we include all the parentheses:
\[
\begin{array}{l@{\quad\;\;}l@{\quad\;\;}l@{\quad\;\;}l@{\quad\;\;}l}
( ( {\ast} \cdot {\ast} ) \cdot {\ast} ) \cdot {\ast}, &  
( ( {\ast} {\ast} ) \cdot {\ast} ) \cdot {\ast}, &  
( ( {\ast} \cdot {\ast} ) {\ast} ) \cdot {\ast}, &  
( ( {\ast} {\ast} ) {\ast} ) \cdot {\ast}, &  
( {\ast} ( {\ast} \cdot {\ast} ) ) \cdot {\ast}, \\  
( {\ast} ( {\ast} {\ast} ) ) \cdot {\ast}, &  
( {\ast} \cdot {\ast} ) \cdot ( {\ast} \cdot {\ast} ), &  
( {\ast} \cdot {\ast} ) \cdot ( {\ast} {\ast} ), &  
( {\ast} {\ast} ) \cdot ( {\ast} {\ast} ), &  
( ( {\ast} \cdot {\ast} ) \cdot {\ast} ) {\ast}, \\  
( ( {\ast} {\ast} ) \cdot {\ast} ) {\ast}, &  
( ( {\ast} \cdot {\ast} ) {\ast} ) {\ast}, &  
( ( {\ast} {\ast} ) {\ast} ) {\ast}, &  
( {\ast} ( {\ast} \cdot {\ast} ) ) {\ast}, &  
( {\ast} ( {\ast} {\ast} ) ) {\ast}, \\  
( {\ast} \cdot {\ast} ) ( {\ast} \cdot {\ast} ), &  
( {\ast} \cdot {\ast} ) ( {\ast} {\ast} ), &  
( {\ast} {\ast} ) ( {\ast} \cdot {\ast} ), &  
( {\ast} {\ast} ) ( {\ast} {\ast} ), &  
{\ast} ( ( {\ast} \cdot {\ast} ) \cdot {\ast} ), \\  
{\ast} ( ( {\ast} {\ast} ) \cdot {\ast} ), &  
{\ast} ( ( {\ast} \cdot {\ast} ) {\ast} ), &  
{\ast} ( ( {\ast} {\ast} ) {\ast} ), &  
{\ast} ( {\ast} ( {\ast} \cdot {\ast} ) ), &  
{\ast} ( {\ast} ( {\ast} {\ast} ) ). 
\end{array} 
\]
The number of distinct association types for a sequence of $n$ arguments
with two associative binary operations, one commutative and one noncommutative, 
is sequence A276277 in the Online Encyclopedia of Integer Sequences (\url{oeis.org}):
\[
1, \; 2, \; 6, \; 25, \; 111, \; 540, \; 2736, \; 14396, \; 77649, \; 427608, \; 2392866, \; 13570386, \; 77815161, \; \dots
\]
Applying all permutations to the arguments, and ignoring duplications which follow from commutativity,
we obtain 27 distinct multilinear monomials in arity 3, ordered as follows:
\[
\begin{array}
{l@{\;\;}l@{\;\;}l@{\;\;}l@{\;\;}l@{\;\;}l@{\;\;}l@{\;\;}l@{\;\;}l}
( a \cdot b ) \cdot c, &  
( a \cdot c ) \cdot b, &  
( b \cdot c ) \cdot a, &  
( a b ) \cdot c, &  
( a c ) \cdot b, &  
( b a ) \cdot c, &  
( b c ) \cdot a, &  
( c a ) \cdot b, &  
( c b ) \cdot a, \\  
( a \cdot b ) c, &  
( a \cdot c ) b, &  
( b \cdot c ) a, &  
( a b ) c, &  
( a c ) b, &  
( b a ) c, &  
( b c ) a, &  
( c a ) b, &  
( c b ) a, \\  
a ( b \cdot c ), &  
b ( a \cdot c ), &  
c ( a \cdot b ), &  
a ( b c ), &  
a ( c b ), &  
b ( a c ), &  
b ( c a ), &  
c ( a b ), &  
c ( b a ).  
\end{array}
\]
Similarly, we obtain 405 distinct multilinear monomials of arity 4.
The number of distinct multilinear monomials with two associative binary operations, 
one commutative and one noncommutative, 
is the sextuple factorials, sequence A011781 in the OEIS:
\[
\prod_{k=0}^{n-1} \, (6k{+}3)
=
1, \; 3, \; 27, \; 405, \; 8505, \; 229635, \; 7577955, \; 295540245, \; 13299311025, \; \dots
\]
Figure \ref{asscomquarelmat} displays the matrix whose row space is the $S_3$-submodule
generated by three quadratic relations:
associativity for $ab$, 
associativity for $a \cdot b$, and 
the relation expressing the reduction of a monomial of the form $(ab) \cdot c$ to a linear combination 
of permutations of the monomial $(a \cdot b) c$. 

\begin{figure}[ht]
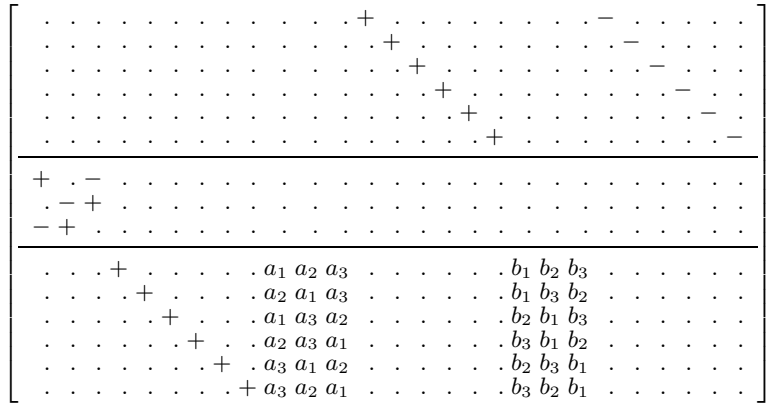

$
\left[
\begin{array}
{
r@{\;}r@{\;}r@{\;}r@{\;}r@{\;}r@{\;}r@{\;}r@{\;}r@{\;}
r@{\;}r@{\;}r@{\;}r@{\;}r@{\;}r@{\;}r@{\;}r@{\;}r@{\;}
r@{\;}r@{\;}r@{\;}r@{\;}r@{\;}r@{\;}r@{\;}r@{\;}r
}
 .& .& .& .& .& .& .& .& .& .& .& .& +& .& .& .& .& .& .& .& .& -& .& .& .& .& .\\[-2pt]
 .& .& .& .& .& .& .& .& .& .& .& .& .& +& .& .& .& .& .& .& .& .& -& .& .& .& .\\[-2pt]
 .& .& .& .& .& .& .& .& .& .& .& .& .& .& +& .& .& .& .& .& .& .& .& -& .& .& .\\[-2pt]
 .& .& .& .& .& .& .& .& .& .& .& .& .& .& .& +& .& .& .& .& .& .& .& .& -& .& .\\[-2pt]
 .& .& .& .& .& .& .& .& .& .& .& .& .& .& .& .& +& .& .& .& .& .& .& .& .& -& .\\[-2pt]
 .& .& .& .& .& .& .& .& .& .& .& .& .& .& .& .& .& +& .& .& .& .& .& .& .& .& -\\\midrule
 +& .& -& .& .& .& .& .& .& .& .& .& .& .& .& .& .& .& .& .& .& .& .& .& .& .& .\\[-2pt]
 .& -& +& .& .& .& .& .& .& .& .& .& .& .& .& .& .& .& .& .& .& .& .& .& .& .& .\\[-2pt]
 -& +& .& .& .& .& .& .& .& .& .& .& .& .& .& .& .& .& .& .& .& .& .& .& .& .& .\\\midrule
 .& .& .& +& .& .& .& .& .&a_1&a_2&a_3& .& .& .& .& .& .&b_1&b_2&b_3& .& .& .& .& .& .\\[-2pt]
 .& .& .& .& +& .& .& .& .&a_2&a_1&a_3& .& .& .& .& .& .&b_1&b_3&b_2& .& .& .& .& .& .\\[-2pt]
 .& .& .& .& .& +& .& .& .&a_1&a_3&a_2& .& .& .& .& .& .&b_2&b_1&b_3& .& .& .& .& .& .\\[-2pt]
 .& .& .& .& .& .& +& .& .&a_2&a_3&a_1& .& .& .& .& .& .&b_3&b_1&b_2& .& .& .& .& .& .\\[-2pt]
 .& .& .& .& .& .& .& +& .&a_3&a_1&a_2& .& .& .& .& .& .&b_2&b_3&b_1& .& .& .& .& .& .\\[-2pt]
 .& .& .& .& .& .& .& .& +&a_3&a_2&a_1& .& .& .& .& .& .&b_3&b_2&b_1& .& .& .& .& .& .
 \end{array}
\right]
$
\vspace{-2mm}
\caption{Associative-commutative quadratic relation matrix}
\label{asscomquarelmat}
\end{figure}

The $S_4$-module generated by the consequences of the three quadratic relations has size
$540 \times 405$.
Its partial Smith form consists of an identity matrix of size 330 and a lower right block
$L$ of size $210 \times 75$.
The matrix $L$ contains 56 distinct nonzero polynomials of degrees 1 and 2; replacing each
by its monic form gives the following 43 polynomials:
\[
\begin{array}{l}
a_3, \; b_2, \; a_2^2, \; a_3^2, \; b_1^2, \; b_2^2, \; a_1 b_3, \; 
a_1 (a_2 {+} 1), \; a_1 (a_1 {+} a_2 {+} a_3 {+} b_1 {+} b_2 {+} b_3), \; a_2 a_1, \; a_2 a_3, \; a_2 b_2, \; 
\\[1pt]
a_2 (a_2 {+} 1), \; a_2 (a_3 {+} b_2), \; 
a_3 a_1, \; a_3 b_1, \; a_3 (a_2 {+} b_1 {+} 1), \; b_1 b_2, \; b_1 b_3, \; b_1 (a_3 {+} b_2), \; b_1 (b_1 {+} 1), \; 
\\[1pt]
b_2 b_3, \; 
b_2 (a_2 {+} b_1 {+} 1), \; b_3 (b_1 {+} 1), \; b_3 (a_1 {+} a_2 {+} a_3 {+} b_1 {+} b_2 {+} b_3), \; 
a_1 (a_2 {-} a_1), \; a_1 (a_3 {-} a_1), \; 
\\[1pt]
a_1 (a_3 {-} a_2), \; 
a_1 (b_2 {-} b_1), \; b_3 (a_3 {-} a_2), \; b_3 (b_2 {-} b_1), \; b_3 (b_3 {-} b_1), \; b_3 (b_3 {-} b_2), \; 
a_2 b_1 {+} a_3^2, \; 
\\[1pt]
a_2 b_1 {+} b_2^2, \; 
a_1 a_2 {+} a_3 b_3, \; a_1 b_1 {+} a_2 b_3, \; a_1 b_2 {+} a_3 b_3, \; a_1 b_2 {+} b_1 b_3, \; 
a_2^2 {+} a_3 b_2 {+} a_2, \; a_3 b_2 {+} b_1^2 {+} b_1, \; 
\\[1pt]
a_1 a_3 {+} a_2 b_3 {+} b_3, \; a_1 b_1 {+} b_2 b_3 {+} a_1.
\end{array}
\]
One easily verifies that the deglex Gr\"obner basis for the ideal generated by these polynomials 
consists of the six variables $a_1, a_2, a_3, b_1, b_2, b_3$
and this completes the proof.
\end{proof}


\section{Distributive laws $\dl\Lie\Ass$}

The methods in this case are very similar to the case $\dl\Com\Ass$
except that instead of a commutative associative operation we have a Lie bracket:
an anticommutative operation satisfying the Jacobi identity.
This requires keeping track of sign changes that occur as a result of anticommutativity 
when calculating normal forms of the monomials in consequences and permutations 
of various quadratic and cubic relations.

\begin{theorem}
The only distributive law  $\dl\Lie\Ass$ is the trivial one.
By Koszul duality, the same conclusion holds for  $\dl\Ass\Com$.
\end{theorem}

\begin{nonexample}
One is tempted to relax the commutativity of the associative multiplication of Poisson algebras, 
keeping other axioms unchanged, as done e.g.~in~\cite{AM}. 
We show that in this case the derivation rule~(\ref{derivationlaw}) does not define 
a distributive law $\dl\Lie\Ass$,
so we suspect that these na\"ive noncommutative Poisson algebras are ill-behaved. 
More specifically, we show that the rule~(\ref{derivationlaw}) 
is not compatible with the anticommutativity of~$[-,-]$. 
Let us consider the equation
\begin{equation}
\label{ten_silenej_Cinan_je_tady}
[ab,cd] = -[cd,ab].
\end{equation}
Expanding its left side using \eqref{derivationlaw} twice gives
\begin{align*}
[ab,cd] &= a[b,cd] +  [a,cd]b =  ac[b,d] +  a[b,c]d + c[a,d]b + [a,c]db,
\end{align*}
while the right side results in
\begin{align*}
-[cd,ab] =  -c[d,ab] - [c,ab]d
&= 
- ca[d,b] - c[d,a]b - a[c,b]d - [c,a]bd
\\
&=
ca[b,d] + c[a,d]b + a[b,c]d + [a,c]bd.
\end{align*}
The compatibility of~(\ref{derivationlaw}) with~(\ref{ten_silenej_Cinan_je_tady}) 
would require the equality
\[
{ac[b,d]}  + c[a,d]b +  a[b,c]d + {[a,c]db} = 
\\
{ca[b,d]} + c[a,d]b + a[b,c]d + {[a,c]bd},
\]
which is the same as
\[
(ac-ca)[b,d] + [a,c](db-bd) = 0.
\]
One however cannot expect this to be true in general unless $ac=ca$ and $db=bd$.
If we denote the commutator of the associative multiplication by $\{-,-\}$ then we obtain
\begin{equation}
\label{Prinutim_se_jet_na_kole?}
\{a,c\}[b,d] = [a,c]\{b,d\},
\end{equation}
which can be found e.g.~in~\cite[Lemma~1.1]{voronov_th} or~\cite[Theorem~1]{vr}.
Theodore Voronov informed us, referring to a rare 1932
book\footnote{Cf. formula~(14), page 41, of the second
  edition~\cite{Fock} of that book.} 
by Fok, that~(\ref{Prinutim_se_jet_na_kole?}) was
first obtained by Dirac, who used it to motivate his argument
that in quantum mechanics, the `quantum Poisson bracket' has 
to be proportional to the commutator of the operators.  
\end{nonexample}

\begin{remark}
\label{Dnes_vecer_s_Andulkou_na_veceri.}
We advise the reader that there are other structures called
`noncommutative Poisson algebras' in the literature.
The structure in \cite{Kubo1,Kubo2} combines Leibniz and associative algebras
via the derivation rule \eqref{derivationlaw}; it is therefore of type  $\dl\Lei\Ass$. 
The structure in \cite{CEEY} is defined as a Poisson algebra on the abelization $A/[A,A]$  
of an associative algebra $A$. 
Other generalizations include 
double Poisson algebras \cite{T,V} equipped with a `double bracket' $A \otimes A \to
A \otimes A$, or a twisted version in the physics paper~\cite{RV}.
\end{remark}


\section{The remaining cases}

In this section we analyze the remaining three types of distributive
laws between the Three Graces.

\begin{theorem}
\label{comcomtheorem}
For $\dl\Com\Com$ we obtain only the trivial distributive law.
\end{theorem}

\begin{proof}
The calculations are similar to those discussed in detail in previous sections, 
so we provide only a brief outline.
The number of distinct association types in arity $n$ for two commutative operations is 
sequence OEIS A226909; see also \cite{BM}:
\[
1, \; 2, \; 4, \; 14, \; 44, \; 164, \; 616, \; 2450, \; 9908, \; 41116, \; 173144, \; 739884, \; 3196344, \; 13944200, \; \dots.
\]
For arities 3 and 4, these types are as follows:
\[
\begin{array}{l}
( * * ) *, \quad  
( * \cdot * ) *, \quad  
( * * ) \cdot *, \quad  
( * \cdot * ) \cdot *;  
\\
( ( * * ) * ) *, \quad  
( ( * \cdot * ) * ) *, \quad  
( ( * * ) \cdot * ) *, \quad  
( ( * \cdot * ) \cdot * ) *, \quad  
( * * ) ( * * ), \quad  
\\
( * * ) ( * \cdot * ), \quad  
( * \cdot * ) ( * \cdot * ), \quad  
( ( * * ) * ) \cdot *, \quad  
( ( * \cdot * ) * ) \cdot *, \quad  
( ( * * ) \cdot * ) \cdot *, \quad  
\\
( ( * \cdot * ) \cdot * ) \cdot *, \quad  
( * * ) \cdot ( * * ), \quad  
( * * ) \cdot ( * \cdot * ), \quad  
( * \cdot * ) \cdot ( * \cdot * ).
\end{array}
\]
The number of distinct multilinear monomials is the quadruple factorials (OEIS A001813): 
\[
\frac{(2n)!}{n!}
= 
1, \; 2, \; 12, \; 120, \; 1680, \; 30240, \; 665280, \; 17297280, \; 518918400, \; 17643225600, \; \dots.
\]
For arity 3, these monomials are as follows (in lex order):
\[
( a b ) c, \;  
( a c ) b, \;  
( b c ) a, \;  
( a \cdot b ) c, \;  
( a \cdot c ) b, \;  
( b \cdot c ) a, \;  
( a b ) \cdot c, \;  
( a c ) \cdot b, \;  
( b c ) \cdot a, \;  
( a \cdot b ) \cdot c, \;  
( a \cdot c ) \cdot b, \;  
( b \cdot c ) \cdot a.
\]
Using these monomials, associativity has the form
\[
(ab)c - (bc)a,
\qquad\qquad
( a \cdot b ) \cdot c - ( b \cdot c ) \cdot a.
\]
The most general distributive law relating the operations is as follows,
where $x_1, x_2, x_3$ are free parameters:
\[
x_1 ( a b ) \cdot c + x_2 ( a c ) \cdot b + x_3 ( b c ) \cdot a - ( a \cdot b ) c.
\]
Applying all permutations of the variables $a, b, c$ to these three relations,
and expressing the relations as row vectors of coefficients, we obtain this matrix:
\[
\left[
\begin{array}{rrrrrrrrrrrr}
 1 &  \cdot & -1 &  \cdot &  \cdot &  \cdot &  \cdot &  \cdot &  \cdot &  \cdot &  \cdot &  \cdot \\ 
 \cdot & -1 &  1 &  \cdot &  \cdot &  \cdot &  \cdot &  \cdot &  \cdot &  \cdot &  \cdot &  \cdot \\ 
-1 &  1 &  \cdot &  \cdot &  \cdot &  \cdot &  \cdot &  \cdot &  \cdot &  \cdot &  \cdot &  \cdot \\ 
 \cdot &  \cdot &  \cdot &  \cdot &  \cdot &  \cdot &  \cdot &  \cdot &  \cdot &  1 &  \cdot & -1 \\ 
 \cdot &  \cdot &  \cdot &  \cdot &  \cdot &  \cdot &  \cdot &  \cdot &  \cdot &  \cdot & -1 &  1 \\ 
 \cdot &  \cdot &  \cdot &  \cdot &  \cdot &  \cdot &  \cdot &  \cdot &  \cdot & -1 &  1 &  \cdot \\ 
 \cdot &  \cdot &  \cdot & x_1 & x_2 & x_3 &  1 &  \cdot &  \cdot &  \cdot &  \cdot &  \cdot \\ 
 \cdot &  \cdot &  \cdot & x_2 & x_3 & x_1 &  \cdot &  \cdot &  1 &  \cdot &  \cdot &  \cdot \\ 
 \cdot &  \cdot &  \cdot & x_3 & x_1 & x_2 &  \cdot &  1 &  \cdot &  \cdot &  \cdot &  \cdot
\end{array}
\right]
\]
We compute the consequences in arity 4 of these nine relations $I$ in arity 3.
If we write $\omega_1$, $\omega_2$ for the two operations then
for each $I$ we obtain $I \circ_k \omega_j$ ($k = 1, 2, 3$; $j = 1, 2$) 
and
$\omega_j \circ_k I$ ($j, k = 1, 2$)
where $\circ_k$ denotes operadic partial composition. 
Each term of each consequence must be straightened using commutativity to convert the
underlying monomial to one of the 120 normal forms in arity 4.
Each quadratic relation $I$ produces 10 cubic consequences for a total of 30; 
applying all permutations of the four variables $a, b, c, d$ we obtain altogether 360
cubic relations, which we store in a $360 \times 120$ matrix $R$ with entries 
0, 1, $-1$, $x_1$, $x_2$, $x_3$. 
Following \cite{BD-book}, we compute a partial Smith form 
\[
\begin{bmatrix} I_{105} & 0 \\ 0 & B \end{bmatrix},
\]
where the lower right block $B$ contains the following nonzero entries:
\[
\begin{array}{l}
x_2^2, \; x_2 x_3, \; x_3 x_1, \; - x_1^2, \; - x_2^2, \; - x_2 x_3, \; - x_3 x_1, \; x_2 - x_3, \; x_3 - x_2, \; - x_3^2 - x_3, \; x_3^2 + x_3,
\\
- x_1 x_2 - x_1, \; x_1 x_2 + x_1, \; - x_3^2 - x_2, \; x_3^2 + x_2, \; - x_2 x_3 - x_3, \; x_2 x_3 + x_3, \; - x_2 x_3 - x_2, 
\\
x_2 x_3 + x_2, \; - x_2 x_3 + x_3^2, \; x_2 x_3 - x_3^2, \; - x_2^2 + x_3^2, \; - x_2^2 + x_2 x_3, \; x_2^2 - x_2 x_3, \; - x_1 x_2 + x_1 x_3,
\\
x_1 x_2 - x_1 x_3, \; - x_1 x_2 - x_1 x_3 - x_1, \; x_1 x_2 + x_1 x_3 + x_1, \; - x_1^2 - x_1 x_2 - x_1 x_3,
\\
- x_1^2 - x_1 x_2 + x_1 x_3, \; - x_1^2 + x_1 x_2 - x_1 x_3, \; x_1^2 - x_1 x_2 + x_1 x_3, \; x_1^2 + x_1 x_2 + x_1 x_3.
\end{array}
\]
The ideal in ${\bfk}[x_1,x_2,x_3]$ generated by these polynomials 
has Gr\"obner basis $x_1$, $x_2$, $x_3$.
\end{proof}

\begin{theorem}
\label{comlietheorem}
For $\dl\Com\Lie$ we obtain only the trivial distributive law.
\end{theorem}

\begin{proof}
Very similar to the proof of Theorem \ref{comcomtheorem}.
\end{proof}

\begin{theorem}
The only nontrivial distributive law $\dl\Lie\Com$ is that for Poisson algebras.
\end{theorem}

The theorem is a particular case of the classification of generalized
distributive laws between $\Lie$ and $\Com$ given in~\cite{BD}. 



\section{Associative-magmatic laws}
\label{Pot_ze_mne_leje.}

In this final section we analyze distributive laws $\dl{\Ass}{\Mag}$
between associative and magmatic (no axioms) multiplications $\bullet$
resp.~$\circ$.

\begin{theorem}
\label{Snad_uz_s_205_rozume_pristavam.}
There are only two non-isomorphic distributive laws between the
associative and magmatic multiplication, the trivial one and the
truncated one represented by the rewrite rules (b), (c), (e) or (f) 
of Theorems~\ref{Pred_odjezdem_do_Ostravy}
and~\ref{vcera_jsem_se_proti_vetru_nadrel}.
\end{theorem}

\begin{proof}
Maple found the following rewrite rules, with $\alpha, \gamma \in
\bbk$ arbitrary parameters for which the square roots in the formulas
exist, and $\iota = \sqrt{-1}$:
\begin{align*}
1) \quad
&
(x \c y) \b z = 0,
\quad
x \b (y \c z) = 0,
\\
2) \quad
&
(x \c y) \b z = 0,
\cr
&
x \b (y\c z) =
\tfrac12 \, ( x \b y ) \c z
+ \tfrac12 \iota \, ( x \b z )\c y
+ \tfrac12 \, y \c( x \b z )
- \tfrac12 \iota \, z\c ( x \b y ),
\\
3) \quad
&
(x\c y) \b z = 
\tfrac12 \, ( x \b z ) \c y
+ \tfrac12 \iota \, ( y \b z )\c x
+ \tfrac12 \, x \c( y \b z )
- \tfrac12 \iota \, y\c ( x \b z ),
\\
&
x \b (y \c z) =0,
\\
4) \quad
&
(x  \c  y) \b z = 0,
\\[-4pt]
&
x \b (y \c z) =
- \gamma \, ( x \b y ) \c z
+ \hksqrt{\gamma^2{+}\gamma} \, ( x \b z )\c y
+ (\gamma{+}1) \, y\c ( x \b z )
- \hksqrt{\gamma^2{+}\gamma} \, z\c  ( x \b y ),
\\
5) \quad
&
(x \c y) \b z = 
- \alpha \, ( x \b z )\c y
+ \hksqrt{\alpha^2{+}\alpha} \, ( y \b z )\c x
+ (\alpha{+}1) \, x \c( y \b z )
- \hksqrt{\alpha^2{+}\alpha} \, y \c( x \b z ),
\\[-2pt]
&
x \b( y \c z) =0.
\end{align*}
Law 1) is the trivial one. 
Laws 4) and 5)  are isomorphic, via the replacement $\bullet \mapsto \bullet^{\rm op}$ 
of the $\bullet$-product by the opposite one. 
Law 2) is obtained from 4) by substituting 
$\gamma = -\frac12$ and,
likewise, the substitution $\alpha =  - \frac12$ brings 5) into 3). 

The proof
will therefore be finished if we show that 4) is isomorphic to the
truncated distributive law. The following method, suggested
by Vladimir Dotsenko,  is based on the substitution
\begin{equation}
\label{Po_navratu_z_Moskvy.}
\gamma = \frac 1{t^2-1}, \ t \not= \pm 1.
\end{equation}
Notice that its inverse  can be written as
\[
t = \frac{\sqrt{\gamma^2 + \gamma}}\gamma,
\]
thus for any $\gamma \not=0$ for which $\sqrt{\gamma^2 + \gamma}$ exists
one has $t$ fulfilling~(\ref{Po_navratu_z_Moskvy.}). It is
straightforward to verify that the replacement
\[
x \circ y \mapsto x \circ y + t ( y \circ x )
\]
brings 4) into the truncated rule
\[
(x  \circ  y) \bullet z = 0,\ x \bullet (y \circ z) = (x \circ y) \bullet z.
\]
If $\gamma=0$ in which case the substitution~(\ref{Po_navratu_z_Moskvy.}) cannot
be used then 4) becomes another (but isomorphic) truncated rule
\[
(x  \circ  y) \bullet z = 0,\ x \bullet (y \circ z) = y \circ (x
\bullet z).
\]
This finishes the proof.
\end{proof}

\begin{remark}
The reason for including the above proof instead of just referring to
the result of Maple calculation was to show that, outside the realm of
Three Graces, various bizarre-looking distributive laws, such as 4) or
5), may exist. Since one of the operads --- in this case 
$\Mag$ --- may have a huge group of automorphisms, these weird laws may however
turn to be isomorphic to mild and expected ones.
\end{remark}

Theorem~\ref{Snad_uz_s_205_rozume_pristavam.} has the following
obvious but surprising 

\begin{corollary}
In the cartesian monoidal category of sets, 
there are no distributive laws  of type $\dl{\Ass}{\Mag}$. 
\end{corollary}


\begin{thebibliography}{99}

\bibitem{AM}
\textsc{A. L. Agore, G. Militaru}:
The global extension problem, crossed products and co-flag non-commutative Poisson algebras. 
\emph{Journal of Algebra} 
426 (2015) 1--31. 

\bibitem{A1}
\textsc{F. Akman}:
On some generalizations of Batalin-Vilkovisky algebras. 
\emph{Journal of Pure and Applied Algebra} 
120 (1997), no.~2, 105--141. 

\bibitem{A2}
\textsc{F. Akman}:
A master identity for homotopy Gerstenhaber algebras. 
\emph{Communications in Mathematical Physics} 
209 (2000), no.~1, 51--76. 

\bibitem{BFFLS1}
\textsc{F. Bayen, M. Flato, C. Fronsdal, A. Lichnerowicz, D. Sternheimer}:
Deformation theory and quantization. I. Deformations of symplectic structures. 
\emph{Annals of Physics} 
111 (1978), no.~1, 61--110. 

\bibitem{BFFLS2}
\textsc{F. Bayen, M. Flato, C. Fronsdal, A. Lichnerowicz, D. Sternheimer}:
Deformation theory and quantization. II. Physical applications. 
\emph{Annals of Physics} 
111 (1978), no.~1, 111--151. 

\bibitem{B}
\textsc{J. Beck}:
Distributive laws.
In: B. Eckmann,
\emph{Seminar on Triples and Categorical Homology Theory},
pages 119--140.
Lecture Notes in Mathematics, 80.
Springer, Berlin-Heidelberg 1969.
Available online:
\url{www.tac.mta.ca/tac/reprints/articles/18/tr18abs.html}

\bibitem{BV}
\textsc{J. M. Boardman, R. M. Vogt}: 
\emph{Homotopy Invariant Algebraic Structures on Topological Spaces}. 
Lecture Notes in Mathematics, 347. 
Springer-Verlag, Berlin-New York, 1973. 

\bibitem{BD-book}
\textsc{M. R. Bremner, V. Dotsenko}: 
\emph{Algebraic Operads: An Algorithmic Companion}. 
CRC Press, Boca Raton, FL, 2016.

\bibitem{BD}
\textsc{M. R. Bremner, V. Dotsenko}: 
Distributive laws for the Lie and Com operads.
Work in progress, October 2017.

\bibitem{BM}
\textsc{M. R. Bremner, S. Madariaga}: 
Lie and Jordan products in interchange algebras. 
\emph{Communications in Algebra} 
44 (2016), no.~8, 3485--3508.

\bibitem{CEEY}
\textsc{X. Chen, A. Eshmatov, F. Eshmatov, S. Yang}:
The derived non-commutative Poisson bracket on Koszul Calabi-Yau algebras. 
\emph{Journal of Noncommutative Geometry} 
11 (2017), no.~1, 111--160. 

 
\bibitem{CG}
\textsc{K. Costello, O. Gwilliam}:
\emph{Factorization Algebras in Quantum Field Theory, Volume 1}. 
New Mathematical Monographs, 31. 
Cambridge University Press, Cambridge, 2017. 

\bibitem{DTT}
\textsc{V. Dolgushev, D. Tamarkin, B. Tsygan}:
The homotopy Gerstenhaber algebra of Hochschild cochains of a regular algebra is formal.
\emph{Journal of Noncommutative Geometry} 
1 (2007), no.~1, 1--25. 

\bibitem{Fock}
\textsc{V. A. Fok}: {\cyr \cyracc Nachala kvantovo\u i mehaniki,}
  Nauka, Moskva, 1976 (2nd edition).

\bibitem{FM}
\textsc{T. F. Fox, M. Markl}:
Distributive laws, bialgebras, and cohomology.
In:
\emph{Operads: Proceedings of Renaissance Conferences (Hartford, CT/Luminy, 1995)},
pages 167--205.
Contemporary Mathematics, 202.
American Mathematical Society, Providence, RI, 1997.

\bibitem{F}
\textsc{J. Francis}:
The tangent complex and Hochschild cohomology of $E_n$-rings. 
\emph{Compositio Mathematica} 
149 (2013), no.~3, 430--480. 

\bibitem{GTV}
\textsc{I. G\'alvez-Carrillo, A. Tonks, B. Vallette}:
Homotopy Batalin-Vilkovisky algebras. 
\emph{Journal of Noncommutative Geometry} 
6 (2012), no.~3, 539--602. 

\bibitem{GV}
\textsc{X. Garc\'ia-Mart\'inez, T. Van der Linden}:
A characterisation of Lie algebras via algebraic exponentiation.
\url{arxiv.org/abs/1711.00689}
(submitted on 2 November 2017).

\bibitem{G}
\textsc{M. Gerstenhaber}: 
The cohomology structure of an associative ring. 
\emph{Annals of Mathematics} 
(2) 78 (1963) 267--288. 
  
\bibitem{Getzler}
\textsc{E. Getzler}:
Batalin-Vilkovisky algebras and two-dimensional topological field theories. 
\emph{Communications in Mathematical Physics} 
159 (1994), no.~2, 265--285. 

\bibitem{ginzburg-kapranov:DMJ94}
\textsc{V. Ginzburg and M.M. Kapranov}:
Koszul duality for operads.
\emph{Duke Mathematical Journal} 
76(1) (1994) 203--272.

\bibitem{H}
\textsc{J. Huebschmann}:
Lie-Rinehart algebras, Gerstenhaber algebras and Batalin-Vilkovisky algebras. 
\emph{Annales de l'Institut Fourier (Grenoble)} 
48 (1998), no.~2, 425--440. 

\bibitem{K1}
\textsc{Y. Kosmann-Schwarzbach}:
From Poisson algebras to Gerstenhaber algebras.
\emph{Annales de l'Institut Fourier (Grenoble)} 
46 (1996), no.~5, 1243--1274. 

\bibitem{K2}
\textsc{Y. Kosmann-Schwarzbach}:
La g\'eom\'etrie de Poisson, cr\'eation du XXe si\`ecle.
[Poisson geometry, a twentieth-century creation]. 
\emph{Sim\'eon-Denis Poisson}, pages 129--172.
Hist. Math. Sci. Phys., Ed. \'Ec. Polytech., Palaiseau, 2013. 

\bibitem{K3}
\textsc{Y. Kosmann-Schwarzbach}:
Les crochets de Poisson, de la m\'ecanique c\'eleste \`a la m\'ecanique quantique.
[Poisson brackets, from celestial to quantum mechanics].
\emph{Sim\'eon-Denis Poisson}, pages 369--401, 
Hist. Math. Sci. Phys., Ed. \'Ec. Polytech., Palaiseau, 2013. 

\bibitem{KM}
\textsc{Y. Kosmann-Schwarzbach, F. Magri}:
Poisson-Nijenhuis structures. 
\emph{Annales de l'Institut Henri Poincar\'e: Physique Th\'eorique} 
53 (1990), no.~1, 35--81. 

\bibitem{Kubo1}
\textsc{F. Kubo}:
Finite-dimensional non-commutative Poisson algebras. 
\emph{Journal of Pure and Applied Algebra} 
113 (1996), no.~3, 307--314. 

\bibitem{Kubo2}
\textsc{F. Kubo}:
Finite-dimensional non-commutative Poisson algebras. II. 
\emph{Communications in Algebra} 
29 (2001), no.~10, 4655--4669. 

\bibitem{Lack}
\textsc{S. Lack}:
Composing PROPS. 
\emph{Theory and Applications of Categories} 
13 (2004), no.~9, 147--163. 

\bibitem{LZ}
\textsc{B. H. Lian, G. J. Zuckerman}:
New perspectives on the BRST-algebraic structure of string theory. 
\emph{Communications in Mathematical Physics} 
154 (1993), no.~3, 613--646. 

\bibitem{LL}
\textsc{M. Livernet, J.-L. Loday}:
The Poisson operad as a limit of associative operads.
Unpublished preprint, March 1998.

\bibitem{loday-vallette}
\textsc{J.-L. Loday and B. Vallette}:
\emph{Algebraic operads}.
Grundlehren der Mathematischen Wissenschaften 
[Fundamental Principles of Mathematical Sciences],
346.
Springer, Heidelberg, 2012.

\bibitem{markl:handbook}
\textsc{M. Markl}:
Operads and PROPs.
In:
\emph{Handbook of algebra. Vol. 5}, pages 87--140.
Elsevier/North-Holland, Amsterdam, 2008.

\bibitem{M}
\textsc{M. Markl}:
Distributive laws and Koszulness.
\emph{Annales de l'Institut Fourier (Grenoble)}
46 (1996), no.~2, 307--323.

\bibitem{markl-shnider-stasheff:book}
\textsc{M. Markl, S. Shnider, and J.D. Stasheff}:
\emph{Operads in Algebra, Topology and Physics}.
Mathematical Surveys and Monographs, 96.
American Mathematical Society, Providence, RI, 2002.
 
\bibitem{MR}
\textsc{M. Markl, E. Remm}:
Algebras with one operation including Poisson and other Lie-admissible algebras.
\emph{Journal of Algebra} 
299 (2006), no.~1, 171--189. 

\bibitem{May}
\textsc{J. P. May}:
\emph{The Geometry of Iterated Loop Spaces}. 
Lecture Notes in Mathematics, Vol. 271. 
Springer-Verlag, Berlin-New York, 1972.

\bibitem{R}
\textsc{C. Roger}:
Gerstenhaber and Batalin-Vilkovisky algebras: algebraic, geometric, and physical aspects. 
\emph{Archivum Mathematicum (Brno)} 
45 (2009), no.~4, 301--324. 

\bibitem{RV}
\textsc{A. E. Ruuge, F. Van Oystaeyen}:
Distortion of the Poisson bracket by the noncommutative Planck constants.
\emph{Communications in Mathematical Physics} 
304 (2011), no.~2, 369--393. 

\bibitem{Shestakov}
\textsc{I. P. Shestakov}:
Quantization of Poisson superalgebras and the specialty of Jordan superalgebras of Poisson type. 
\emph{Algebra i Logika} 32 (1993), no. 5, 571--584, 587 (1994); 
translation in \emph{Algebra and Logic} 32 (1993), no. 5, 309--317 (1994).

\bibitem{S1}
\textsc{D. Sinha}:
Operads and knot spaces. 
\emph{Journal of the American Mathematical Society} 
19 (2006), no.~2, 461--486. 

\bibitem{S2}
\textsc{D. Sinha}:
The (non-equivariant) homology of the little disks operad.
\emph{OPERADS 2009}, pages 253--279.
S\'eminaires et Congr\`es, 26.
Soci\'et\'e Math\'ematique de France, Paris, 2013. 

\bibitem{Street}
\textsc{R. Street}:
The formal theory of monads. 
\emph{Journal of Pure and Applied Algebra} 
2 (1972), no.~2, 149--168. 

\bibitem{T}
\textsc{V. Turaev}:
Poisson-Gerstenhaber brackets in representation algebras. 
\emph{Journal of Algebra} 
402 (2014) 435--478. 

\bibitem{V}
\textsc{M. Van den Bergh}:
Double Poisson algebras. 
\emph{Transactions of the American Mathematical Society} 
360 (2008), no.~11, 5711--5769. 

\bibitem{voronov_th}
\textsc{T.~Voronov}:
Graded manifolds and Drinfel'd doubles for Lie bialgebroids. 
\emph{Contemporary Mathematics} 
315 (2002) 131--168.

\bibitem{vr}
\textsc{F. F. Voronov}:
On the Poisson hull of a Lie algebra: a ``noncommutative'' moment space. 
\emph{Funktsional'ny{\u\i} Analiz i ego Prilozheniya} 
29 (1995) no.~3, 61--64.

\bibitem{X}
\textsc{P. Xu}:
Gerstenhaber algebras and BV-algebras in Poisson geometry. 
\emph{Communications in Mathematical Physics} 
200 (1999), no.~3, 545--560. 

\bibitem{encyclopedia}
\textsc{G. W. Zinbiel}:
Encyclopedia of types of algebras 2010.
\emph{Operads and Universal Algebra}, pages 21--297.
Nankai Series in Pure, Applied Mathematics and Theoretical Physics, 9. 
World Scientific, Hackensack, 2012. 

\end{thebibliography}
\end{document}